\documentclass[a4paper,reqno]{amsart}
\usepackage{amsmath}
\usepackage{amssymb}
\usepackage{amsthm}
\usepackage{microtype}
\usepackage{mathrsfs}
\usepackage{graphicx}
\usepackage{textcomp}
\usepackage{lmodern}
\usepackage{mathtools}
\mathtoolsset{showonlyrefs}

\newtheorem{thm}{Theorem}[section]
\newtheorem{prop}[thm]{Proposition}
\newtheorem{lemma}[thm]{Lemma}

\theoremstyle{remark}
\newtheorem{assump}[thm]{Assumption}
\newtheorem{remark}[thm]{Remark}
\theoremstyle{definition}

\newtheorem{defn}[thm]{Definition}

\numberwithin{equation}{section}

\def\Ham{\mathcal{H}}

\def\dG{\Ham_{\text{dG}}}
\def\M{\Ham_{\text{M}}}
\def\Ang{\mathcal{A}}
\def\h{\mathfrak{h}}
\def\q{\mathfrak{q}}

\def\eig#1#2{\lambda_{#1,#2}}
\def\eigone#1{\eig{1}{#1}}

\def\tpsi{\Psi}

\def\Ro{R_{o}}
\def\Ri{R_{i}}

\def\R{\mathbb{R}}
\def\Oh{\mathcal{O}}

\let\phi=\varphi
\let\epsilon=\varepsilon
\DeclareMathOperator{\curl}{curl}
\DeclareMathOperator{\Div}{div}
\DeclareMathOperator{\dist}{dist}
\DeclareMathOperator{\dom}{Dom}
\DeclareMathOperator{\spec}{Spec}

\title{Lack of diamagnetism and the Little--Parks effect}

\author{S\o ren Fournais}
\author{Mikael Persson Sundqvist}
\address[S\o ren Fournais]{Aarhus University, Department of
  Mathematics, Ny Munkegade 118, 8000 Aarhus C, Denmark}
\email{fournais@imf.au.dk}
\address[Mikael Persson Sundqvist]{Lund University, Department of Mathematical 
Sciences, Lund, Sweden}
\email{mickep@maths.lth.se}

\begin{document}

\begin{abstract}
When a superconducting sample is submitted to a sufficiently strong external 
magnetic field, the superconductivity of the material is lost. In this paper we 
prove that this effect does not, in general, take place at a unique value of the 
external magnetic field strength. Indeed, for a sample in the shape of a narrow 
annulus the set of magnetic field strengths for which the sample is 
superconducting is not an interval. This is a rigorous justification of the 
Little--Parks effect.
We also show that the same oscillation effect can happen for disc-shaped samples 
if the external magnetic field is non-uniform. In this case the oscillations 
can even occur repeatedly along arbitrarily large values of the Ginzburg--Landau 
parameter $\kappa$.
The analysis is based on an understanding of the underlying spectral theory for 
a magnetic Schr\"{o}dinger operator. It is shown that the ground state energy 
of such an operator is not in general a monotone function of the intensity of 
the field, even in the limit of strong fields.

\end{abstract}

\maketitle

\section{Introduction}
\subsection{Discussion}
We will consider the Ginzburg--Landau model of superconductivity. If a 
$2$-dimensional superconducting sample with Ginzburg--Landau parameter $\kappa$ 
is submitted to a uniform magnetic field of strength $\sigma$, then (by a 
theorem of Giorgi and Phillips~\cite{giph}) there exists a field strength 
$\overline{H_{C_3}(\kappa)}$ such that if 
$\sigma > \overline{H_{C_3}(\kappa)}$, 
then the sample will be in its normal state, i.e. superconductivity is lost 
altogether. It is at first sight natural to expect this phenomenon to mark a 
monotone transition, i.e. to expect that the material is in its superconducting 
(possibly mixed) state for all $\sigma<\overline{H_{C_3}(\kappa)}$.

Indeed, such a monotonicity result has been proved recently in a number of 
geometric situations and in both $2$ and $3$ dimensional 
settings~\cite{fohe3,fohe5,fohe6,fope} in the case where the Ginzburg--Landau 
parameter $\kappa$ is large (it also follows from asymptotic expansions obtained 
in other works such as \cite{MR2496304,MR3073418}). However, Nature does not 
support this monotonicity 
in general. The famous Little--Parks effect~\cite{lipa} shows that for narrow 
cylinders (or $2$D annuli) one has an oscillatory behavior instead of 
monotonicity.\footnote{In connection to the Little--Parks effect one often 
discusses the (solid) disc as another example, where the effect of surface 
superconductivity provides a localization to the boundary and therefore 
effectively introduces non-trivial topology which should give oscillations. 
However, as already the early studies of Saint-James~\cite{saja} show 
(see also~\cite{fohe5}), in the case of the solid disc these oscillations are 
superposed on a linear background and are not strong enough to break the 
monotonicity of the background.}

In this paper we will establish such `oscillatory' effects rigorously in 
different geometric settings.

The lack of monotonicity comes from the topology/geometry of the annulus. It is 
natural 
to ask whether one can get such an oscillatory effect for (non-vanishing) 
magnetic fields defined on domains without topology. From the previous 
investigations \cite{fohe3} we know this to be impossible for a uniform 
magnetic field, but how about more general fields?
The analysis of constant magnetic fields tells us that this question is 
linked to a 
purely spectral problem, namely whether the first eigenvalue of the 
Schr\"{o}dinger operator $(-i\nabla + B {\mathbf F})^2$ is monotone 
increasing in the parameter (strength of the magnetic field) $B$ for 
sufficiently large values of $B$. This property has been called `strong 
diamagnetism' and has been proved for large classes of magnetic fields---it is 
even `generically' satisfied \cite{fohe3,fohe5,fohe6,fope,MR2496304,MR3073418}. 
However, we produce counterexamples in the general case.

\subsection{Ginzburg--Landau theory}
The Ginzburg--Landau theory of superconductivity is based on the energy 
functional
\begin{align}\label{eq:DefGL}
\mathcal{G}_{\kappa,\sigma}(\psi,\mathbf{A})
&=\int_{\Omega} |(-i\nabla+\kappa\sigma \mathbf{A})\psi|^2
-\kappa^2|\psi|^2+\frac{\kappa^2}{2}|\psi|^4\,dx\nonumber \\
&\quad+(\kappa\sigma)^2\int_{\widetilde{\Omega}} |\curl \mathbf{A}-\beta|^2\,dx.
\end{align}
Here $\kappa > 0$ is a material parameter (the Ginzburg--Landau parameter), 
$\sigma\geq 0$ is a parameter measuring the intensity of the external magnetic 
field. The domain $\Omega \subseteq {\mathbb R}^2$ is the part of space occupied 
by the superconducting material. For $\widetilde{\Omega}$ there are two natural 
choices. One can take $\widetilde{\Omega} = {\mathbb R}^2$. That will not be our 
choice here because for reasons of simplicity we want to avoid an unnecessary 
technical complication connected with unbounded domains in ${\mathbb R}^2$ 
(for details on how to handle this issue see \cite{giti,gism}). One 
can also---and that will be our convention here---take $\widetilde{\Omega}$ to 
be the smallest simply connected domain containing $\Omega$, i.e. the union of 
$\Omega$ and all the `holes' in $\Omega$. The function 
$\beta \in L^2(\widetilde{\Omega})$ is the profile of the external magnetic 
field.

In the setting of bounded $\Omega \subset {\mathbb R}^2$ the functional 
$\mathcal{G}_{\kappa,\sigma}$ is naturally defined on 
$(\psi, \mathbf{A}) \in H^1(\Omega) 
\times H^1(\widetilde{\Omega}, {\mathbb R}^2)$.
The functional is immediately seen to be gauge invariant, 
$\mathcal{G}_{\kappa,\sigma}(\psi,\mathbf{A}) 
= \mathcal{G}_{\kappa,\sigma}(\psi e^{-i \kappa \sigma \phi} ,\mathbf{A}
+\nabla \phi)$. 
The vector field ${\mathbf A}$ models the induced magnetic vector potential. 
The function $\psi$ measures the superconducting properties of the material, 
with $|\psi(x)|$ being a measure of the local density of Cooper pairs.

We say that a minimizer $(\psi, \mathbf{A})$ of the Ginzburg--Landau functional 
is trivial if 
$\psi\equiv 0$ and $\curl\mathbf{A}=\beta$. 
In each of the situations we will encounter, the notation $\mathbf{F}$ will be 
reserved for a fixed choice of vector potential with $\curl \mathbf{F} = \beta$.
For trivial minimizers we clearly have
$\mathcal{G}_{\kappa,\sigma}(\psi,\mathbf{A})=0$. For a nontrivial
minimizer the functional must be negative, since one gets from the 
Euler-Lagrange equations of a minimizer that 
\[
\mathcal{G}_{\kappa,\sigma}(\psi,\mathbf{A}) = 
-\frac{\kappa^2}{2} \| \psi \|_4^4,
\]
if $(\psi, \mathbf{A})$ is a minimizer.

We define the set
\[
\mathcal{N}(\kappa):=
\{\sigma>0~\mid~\mathcal{G}_{\kappa,\sigma}
\text{ has a nontrivial minimizer 
$(\psi,\mathbf{A})$}\}.
\]
Following \cite{lupa1} one typically defines the third critical field to be 
given by $\sup \mathcal{N}(\kappa)$, which is finite by \cite{giph}. However, 
unless $\mathcal{N}(\kappa)$ is an interval, this definition is not the only 
natural one to take---see \cite{fohe3,fohebook} for a discussion. We will see below 
that this is not always the case.

\subsection{Oscillations in the third critical field}
Let $\Omega=\{x\in\mathbb{R}^2~\mid~\Ri<|x|<\Ro\}$ denote the annulus with 
inner radius $\Ri$ and outer radius $\Ro$, let $\beta \equiv 1$.
In this case we will write $D = \widetilde{\Omega} = B(0,\Ro)$ i.e. the disc of 
radius $\Ro$.

\begin{thm}\label{thm:main3}
There exists an annulus $\Omega$ and a $\kappa_0>0$ such the set 
$\mathcal{N}(\kappa_0)$ is not an interval.
\end{thm}

\begin{remark}
The mechanism behind this result is a convergence of the magnetic quadratic 
form on the annulus to the corresponding form on the circle. This convergence 
was already noticed in the works \cite{BR,RS}, where also `annuli' of 
non-uniform width were considered. It is likely that one could deduce 
Theorem~\ref{thm:main3} from these works, however, we prefer to give a simple 
independent proof which also emphasizes the connection to the 
Bohm--Aharonov-effect.
\end{remark}

\begin{remark}
\label{rem:largerkappa}
By shrinking the inner radius $\Ri$ of the annulus, we can get $\kappa_0$ as 
large as we want, since the eigenvalues of the limiting problem will then cross 
at a level $1/(2\Ri)^2$. In particular it is possible to have 
$\kappa_0>1/\sqrt{2}$, which means that Theorem~\ref{thm:main3} also applies 
to superconductors of Type II.
\end{remark}

One may criticize the result of Theorem~\ref{thm:main3} on two accounts. One 
could desire not to have the topology fixed a priori, but rather have it 
generated by localization properties of the minimizer. Also most previous 
mathematical analysis has considered the limit of large values of $\kappa$. 
One can show that for sufficiently large values of $\kappa$ the set $\mathcal{N}(\kappa)$ of a superconducting 
sample in the shape of an annulus will behave as the one of the disc with the same outer 
radius, and it is known that for the disc and with constant magnetic field---for sufficiently large values of 
$\kappa$---$\mathcal{N}(\kappa)$ is indeed an interval \cite{fohe3}.

Our next theorem remedies these defects.

\begin{thm}\label{thm:HC3-oscillation}
Let $\Omega$ be the unit disc in $\R^2$. There exists an everywhere positive 
magnetic field $\beta(x)$ such that for all $\kappa_0>0$ there exists 
$\kappa > \kappa_0$ satisfying that $\mathcal{N}(\kappa)$ is not an interval.
\end{thm}

In fact, the magnetic field can be chosen as $\beta(x) = \delta+(1-|x|)^2$, 
where $\delta>0$ is some sufficiently small constant.
Theorem~\ref{thm:HC3-oscillation} follows directly from Theorem~\ref{thm:main} 
(or Theorem~\ref{thm:intasymptot}) below using \cite[Prop. 13.1.7]{fohebook}.
Actually, it easily follows from Theorem~\ref{thm:intasymptot} below, that for 
all integers $n>0$ we can choose $\delta$ so small that ${\mathcal N}(\kappa)$ 
will consist of at least $n$ intervals for all $\kappa$ sufficiently large.

\subsection{Lack of strong diamagnetism}\mbox{}\par

For easy reference we collect the notation and assumptions concerning the magnetic fields that we will treat.
We will work on an open set $\Omega$ being one the following three cases $\Omega \in \{ {\mathbb R}^2, B(0,1),  {\mathbb R}^2 \setminus \overline{B(0,1)}\}$.

\begin{assump}\label{ass:Magnetic}
Suppose that $\beta(x) = \tilde{\beta}(|x|) \in L^{\infty}_{\text{loc}}(\Omega)$, 
is a non-negative, radial magnetic field, possessing five continuous derivatives
in an open neighborhood $U$ of the unit circle $\{x\in\R^2: |x|=1\}$.
Define
\begin{align}
\delta := \tilde{\beta}(1) \geq 0,
\end{align}
and assume that $\tilde\beta'(1)=0$ and write
\begin{align}
\tilde{\beta}''(1) =:k .
\end{align}
When $\Omega \in \{ B(0,1),\  {\mathbb R}^2 \setminus \overline{B(0,1)}\}$, 
we assume that
\begin{align}
\Theta_0 \delta < \inf_{x\in \Omega} \beta(x),
\end{align}
where $\Theta_0 <1$ is the spectral constant recalled in Appendix~\ref{sec:dG}. 
When $\Omega = {\mathbb R}^2$, we impose the stronger assumption that 
$\tilde{\beta}(r)$ has a unique, non-degenerate minimum at $r=1$ and that
\begin{align}
\inf_{x \in {\mathbb R}^2 \setminus U} \beta(x) > \delta.
\end{align}
\end{assump}

\begin{remark}
The assumptions assure that ground state eigenfunctions will be localized near $r=1$. For $\Omega = {\mathbb R}^2$, we have $k>0$ by assumption, but that is not necessarily true in the cases with boundary.
\end{remark}

\begin{defn}
We define
\begin{align}
\Phi := \frac{1}{2\pi} \int_{\{ |x| <1 \}} \beta(x)\,dx = \int_0^1 \tilde{\beta}(r) r\,dr,
\end{align}
i.e. $\Phi$ denotes the magnetic flux through the unit disc.
\end{defn}

For a magnetic field satisfying Assumption~\ref{ass:Magnetic} and $B>0$, 
we study the lowest eigenvalue $\eigone{\Ham(B)}$ of the self-adjoint 
magnetic Schr\"{o}dinger operator
\[
\Ham(B)=(-i\nabla+B\mathbf{F})^2
\]
in $L^2(\Omega)$. Here $\mathbf{F}$ is a magnetic vector potential associated
with the magnetic field $\beta$. We refer the reader to Section~\ref{sec:prel} 
for a more complete definition of this operator and the eigenvalue. 

We will study this eigenvalue problem in three cases, namely for 
$\Omega$ the unit disc, the complement of the unit disc 
and the whole plane $\R^2$. If $\Omega$ has a non-empty boundary we impose a 
magnetic Neumann boundary condition.

The next theorem states that if $\Omega$ is the unit disc or its 
complement,
then special choices of magnetic fields satisfying Assumption~\ref{ass:Magnetic} will give that the
function $B\mapsto \eigone{\Ham(B)}$ is \emph{not} monotonically increasing 
for large $B$. Before stating the theorems, we remind the reader that
\[
\xi_0,\quad \Theta_0,\quad \text{and}\quad \phi_{\xi_0}(0)
\]
are universal (spectral) constants coming from the de Gennes model operator---this is 
recalled in Appendix~\ref{sec:dG}.

\begin{thm}
\label{thm:main}
Let $\Omega$ be the unit disc or its complement. Suppose that $\beta$ satisfies Assumption~\ref{ass:Magnetic}.
Assume that $\delta>0$ and
\begin{align}\label{eq:fluxcondition}
\Phi > \frac{\Theta_0}{\xi_0 \varphi_{\xi_0}(0)^2} \delta.
\end{align}
Then for all $B_0>0$ there exist $B_1$ and 
$B_2$, with $B_0<B_1<B_2$, such that
\[
\eigone{\Ham(B_1)}>\eigone{\Ham(B_2)}.
\]
On the other hand, if 
\begin{align}\label{eq:fluxcondition2}
\Phi < \frac{\Theta_0}{\xi_0 \varphi_{\xi_0}(0)^2} \delta.
\end{align}
Then there exists $B_0>0$ such that $B \mapsto \eigone{\Ham(B)}$ is monotone increasing on $[B_0, \infty)$.
\end{thm}

\begin{remark}
In particular, \eqref{eq:fluxcondition} holds for the magnetic field
\begin{align}
\beta(x) = \delta + (1-|x|)^2,
\end{align}
for all $\delta>0$ sufficiently small---the flux in this case is 
$\Phi = \frac{\delta}{2} + \frac{1}{12}$. Therefore, this magnetic field will 
not display monotonicity for large field strength.
\end{remark}

Theorem~\ref{thm:main} is a consequence of the following precise asymptotic 
formulas for the ground state eigenvalue given as Theorem~\ref{thm:extasymptot} 
and Theorem~\ref{thm:intasymptot}.

\begin{thm}
\label{thm:extasymptot}
Suppose that $\Omega$ is the complement of the unit disc, that $\beta$ satisfies Assumption~\ref{ass:Magnetic} and that
$\delta>0$. 
Then there are constants $C_0^{\text{ext}}$ and $C_1^{\text{ext}}$ such that if
\[
\Delta_B^{\text{ext}}
:=\inf_{m\in\mathbb{Z}}\bigl|
m-\Phi B-\xi_0 (\delta B)^{1/2}
-C_0^{\text{ext}}\bigr|,
\]
then, as $B\to+\infty$,
\[
\eigone{\Ham(B)} = \Theta_0\delta B 
+ \frac{1}{3}\phi_{\xi_0}(0)^2 (\delta B)^{1/2} 
+ \xi_0\,\phi_{\xi_0}(0)^2
\bigl((\Delta_B^{\text{ext}})^2+C_1^{\text{ext}}\bigr)+\Oh(B^{-1/2}).
\]
\end{thm}

\begin{remark}
By a careful reading of the proof, one will realize that the constant 
$C_0^{\text{ext}}$ is independent of $\delta$ but that $C_1^{\text{ext}}$ 
depends on $\delta$. However, for our purposes this extra information is 
irrelevant.
\end{remark}

A similar expansion holds in the interior of the unit disc.
\begin{thm}
\label{thm:intasymptot}
Suppose that $\Omega$ is the unit disc,  that $\beta$ satisfies Assumption~\ref{ass:Magnetic} and that
$\delta>0$. 
Then there exist constants $C_0^{\text{int}}$ and $C_1^{\text{int}}$ such that 
if
\[
\Delta_B^{\text{int}}
:=\inf_{m\in\mathbb{Z}}\bigl|m-\Phi B+\xi_0 (\delta B)^{1/2}
-C_0^{\text{int}}\bigr|,
\]
then, as $B\to+\infty$,
\[
\eigone{\Ham(B)} = \Theta_0\delta B 
-\frac{1}{3}\phi_{\xi_0}(0)^2 (\delta B)^{1/2}
+ \xi_0\,\phi_{\xi_0}(0)^2\bigl((\Delta_B^{\text{int}})^2
+C_1^{\text{int}}\bigr)+\Oh(B^{-1/2}).
\]
\end{thm}

\begin{remark}
Notice that for the disc or its complement, the constant magnetic field 
$\beta(x) = \delta >0$ satisfies Assumption~\ref{ass:Magnetic}, so 
Theorems~\ref{thm:extasymptot} and~\ref{thm:intasymptot} imply this special 
case. This agrees with the calculations in~\cite{fohe3} (see 
also~\cite{fohebook}). In the case of constant field~\eqref{eq:fluxcondition} 
is not satisfied, and one does get monotonicity of the ground state energy for 
large magnetic field (this is discussed in detail in~\cite{fohe3}).
\end{remark}

We continue with $\Omega=\R^2$. Here, we are only able to destroy monotonicity
in the case $\delta=0$.

\begin{thm}\label{thm:wholeplane}
Let $\Omega=\R^2$. Then, for all $\delta >0$ and all magnetic fields satisfying Assumption~\ref{ass:Magnetic} there exists a $B_0>0$ such that
$\eigone{\Ham(B)}$ is monotonically increasing for $B>B_0$. However, if 
$\delta =0$, then $B \mapsto \eigone{\Ham(B)}$ is not monotone increasing on 
any unbounded half-interval.
\end{thm}

As for the disc and the exterior of the disc, the proof of this result goes
via asymptotic expansions.

\begin{thm}\label{thm:R2deltapos}
Suppose that $\Omega=\R^2$, and that $\beta$ satisfies Assumption~\ref{ass:Magnetic} with $\delta>0$. Then, as $B\to+\infty$,
\[
\eigone{\Ham(B)} = \delta B+\frac{k}{4\delta} + \Oh(B^{-1/2}).
\]
\end{thm}

\begin{thm}\label{thm:R2deltazero}
Let $c_0>0$ and $\Xi$ be the spectral constants from~\eqref{eq:c0} 
and~\eqref{eq:DefXi} respectively. Suppose that 
$\Omega=\R^2$, and that and that $\beta$ satisfies Assumption~\ref{ass:Magnetic} 
with $\delta=0$. There exist constants $C_1$ and $C_2$
such that if
\[
\Delta_B := \inf_{m\in\mathbb{Z}}\bigl|m-\Phi B-C_1\bigr|,
\]
then, as $B\to+\infty$,
\[
\eigone{\Ham(B)} = \Bigl(\frac{k}{2}\Bigr)^{1/2}\Xi B^{1/2}+\frac{c_0}{2}\bigl(\Delta_B^2+C_2\bigr) + o(1).
\]
\end{thm}

\begin{remark}
In all of the results above the ground state has angular momentum 
$m \approx \Phi B$ (to leading order in $B$).
We recall that $\Phi B$ is the total flux through the unit 
disc---the bounded domain enclosed by the curve where we have localization.
The possibility to obtain non-monotonicity comes from the condition that $m$
must be an integer, which leads to frustration. 
This is similar to examples in~\cite{Er1}. 
\end{remark}

\begin{remark}
Theorem~\ref{thm:wholeplane} raises the question whether one can break strong 
diamagnetism with a strictly positive magnetic field on the whole plane.
\end{remark}

\subsection{Organization of the paper}
In the next section we define the operators involved and perform the Fourier
decomposition reducing the study to a family of ordinary differential operators.

In Section~\ref{sec:annulus} we prove a non-monotonicity result for an annulus
and use that to prove Theorem~\ref{thm:main3}. In Section~\ref{sec:section4} we
work in the exterior of the unit disc and prove Theorem~\ref{thm:extasymptot}.
We indicate in Section~\ref{sec:disc} how the proof of 
Theorem~\ref{thm:extasymptot} can be modified to give the proof of
Theorem~\ref{thm:intasymptot}. In Section~\ref{sec:nonmondiscanddiscext} we see
how Theorem~\ref{thm:extasymptot} and Theorem~\ref{thm:intasymptot} imply 
Theorem~\ref{thm:main}.

In Section~\ref{sec:planeg0} we prove Theorem~\ref{thm:R2deltapos} and in
Section~\ref{sec:plane0} we prove Theorem~\ref{thm:R2deltazero}. These two
results are used to prove Theorem~\ref{thm:wholeplane}.

\section{Preliminaries}
\label{sec:prel}

\subsection{Definition of the operator}
\label{sec:defs}

We consider the self-adjoint magnetic Neumann Schr\"odinger operator
\begin{equation}\label{eq:neumannop}
\Ham(B)=(-i\nabla+B\mathbf{F})^2
\end{equation}
with domain
\begin{align}\label{eq:neumanncond-New}
\dom(\Ham(B))=\bigl\{\psi\in L^2(\Omega)&\mid
(-i\nabla+B\mathbf{F})^2 \psi \in L^2(\Omega) \nonumber \\
&\quad \text{ and }
N(x)\cdot(-i\nabla+B\mathbf{F})\psi|_{\partial\Omega}=0\bigr\}.
\end{align}

Here $N(x)$ is the interior unit normal to $\partial\Omega$, 
\[
\beta(x)=\Bigl(\frac{\partial F_2}{\partial x_1}
-\frac{\partial F_1}{\partial x_2}\Bigr),\quad \mathbf{F}=(F_1,F_2),
\]
and $B\geq 0$ is the strength of the magnetic field. 

In general, for a self-adjoint operator $\Ham$ that is semi-bounded from below 
we will write
\begin{equation*}
\eigone{\Ham} = \inf\spec\bigl(\Ham\bigr)
\end{equation*}
for the lowest point of the spectrum of $\Ham$. 

In the case of the disc or if $\beta(x)\to+\infty$ as $|x|\to+\infty$
the operator has compact resolvent (see~\cite{avhesi}). If $\Omega$ is
unbounded and if $\beta(x)\not\to+\infty$, then the essential spectrum
will be bounded below by $\liminf_{r\to+\infty} B\tilde{\beta}(r)>B\delta$
(see~\cite{hemo88} for the case of $\R^2$ and \cite{kape} for the case
of the exterior of the disc). In any case, as it will follow by the results
below, $\eigone{\Ham(B)}$ will be an eigenvalue.

\subsection{Fourier decomposition}
We will work in domains $\Omega$ that are rotationally symmetric. For that 
reason, we will often work in polar coordinates
\begin{equation*}
\left\{
\begin{aligned}
x_1 &= r \cos\theta,\\
x_2 &= r \sin\theta,
\end{aligned}\right.\qquad r\in I,\ 0\leq\theta<2\pi.
\end{equation*}
Here $I\subset[0,+\infty)$ will be an interval. 

Moreover, we will work with magnetic fields that depends only on $r=|x|$.

For a radial magnetic field $\beta(x)=\tilde\beta(r)$ we will work
with the gauge
\[
\mathbf{F}(x)=a(r)(-\sin\theta,\cos\theta),
\]
where\footnote{Notice that $\int_0^r \tilde\beta(s)s\,ds 
= \frac{1}{2\pi} \int_{B(0,r)} \beta(x)\,dx$, so $ra(r)$ has an immediate 
interpretation in terms of the flux through the disc of radius $r$.}
\begin{align}
\label{eq:Def_a}
a(r)=\frac{1}{r}\int_0^r \tilde\beta(s) s\,ds.
\end{align}
In calculations, we will often meet the expression $(\frac{m}{r} - B a(r))^2$. 
This we can write as
\begin{align}\label{eq:Cancellation}
\Bigl(\frac{m}{r} - B a(r)\Bigr)^2 &= \frac{1}{r^2} \bigl(m - B r a(r)\bigr)^2,
\end{align}
where
\begin{align}
\label{eq:rar-split}
r a(r) = \int_0^1 \tilde{\beta}(s) s \,ds + \int_1^r \tilde{\beta}(s) s \,ds = \Phi + \int_0^{r-1} \tilde{\beta}(1+s) (1+s) \,ds.
\end{align}
Thus, under Assumption~\ref{ass:Magnetic}, as $r\to 1$,

\begin{align}\label{eq:rar_expansion}
ra(r)&=\Phi + \delta(r-1) + \frac{\delta}{2}(r-1)^2 + \frac{k}{6}(r-1)^3 
+\Bigl(\frac{c}{24}+\frac{k}{8}\Bigr)(r-1)^4+\mathcal{O}((r-1)^5).
\end{align}
with $c= \tilde{\beta}'''(1)$.

The expression for the operator $\Ham(B)$ in polar coordinates becomes
\[
\Ham(B)=-\frac{\partial^2}{\partial r^2}
-\frac{1}{r}\frac{\partial}{\partial r}
+\Bigl(\frac{i}{r}\frac{\partial}{\partial\theta}-Ba(r)\Bigr)^2.
\]
We decompose the Hilbert space as (Here $I$ denotes any of the intervals 
$(\Ri,\Ro)$, $(0,1)$, 
$(1,+\infty)$ or $(0,+\infty)$)
\begin{align*}
L^2(\Omega)\cong 
L^2\bigl(I, rdr \bigr) \otimes L^2(\mathbb{S}^1,d\theta)
\cong \bigoplus_{\mathclap{m=-\infty}}^\infty L^2\bigl(I, rdr \bigr) 
\otimes \frac{e^{-im\theta}}{\sqrt{2\pi}},
\end{align*}
that is, for a function $\psi\in L^2(\Omega)$, we write
\begin{equation*}
\psi(r,\theta)=\sum_{m\in\mathbb{Z}} \psi_m(r)
\frac{e^{-im\theta}}{\sqrt{2\pi}},
\end{equation*}
where $\psi_m\in L^2\bigl(I, r dr \bigr)$. Next, we write the operator 
$\Ham(B)$ corresponding to this decomposition as
\begin{equation*}
\Ham(B) = \bigoplus_{\mathclap{m=-\infty}}^\infty \Ham_m(B)\otimes 1,
\end{equation*}
where $\Ham_m(B)$ is the self-adjoint operator acting in 
$L^2\bigl(I, r\,dr\bigr)$, given by
\begin{equation*}
\Ham_m(B)
=-\frac{d^2}{dr^2}-\frac{1}{r}\frac{d}{dr}
+\Bigl(\frac{m}{r}-Ba(r)\Bigr)^2,
\end{equation*}
with Neumann boundary 
conditions at the endpoints of $I$. The quadratic form 
corresponding to $\Ham_m(B)$ is given by
\begin{equation}
\label{eq:quad}
\q_m[\psi]=\int_I \Bigl[|\psi'(r)|^2
+\Bigl(\frac{m}{r}-Ba(r)\Bigr)^2|\psi(r)|^2 
\Bigr]r\, dr.
\end{equation}
It holds that
\begin{equation}
\label{eq:infmeig}
\eigone{\Ham(B)} = \inf_{m\in\mathbb{Z}}\eigone{\Ham_m(B)}.
\end{equation}

\section{The analysis of the annulus}
\label{sec:annulus}
\subsection{Introduction}
In this section we will let
\[
\beta(x)=1 \quad \text{and}\quad 
\Omega=\bigl\{x\in\mathbb{R}^2~\mid~\Ri<|x|<\Ro\bigr\}.
\]
We aim to prove Theorem~\ref{thm:main3}.

\subsection{The linear result}
We first notice the non-monotonicity of the function 
$B\mapsto \eigone{\Ham(B)}$.

\begin{thm}
\label{thm:mainlin}
Let $\Ri=1$ and $1<\Ro<\sqrt{2}$. Then the operator $\Ham(B)$ in the annulus 
$\Omega$ satisfies
\[
\left.\frac{d}{d B}\eigone{\Ham(B)}\right|_{B=1}<0.
\]
In particular, the function $B\mapsto \eigone{\Ham(B)}$ is monotonically 
decreasing around ${B=1}$.
\end{thm}

One might suspect that some properties of $\Ham(B)$ are carried over to some
model problem on the circle, as $\Ro\searrow\Ri$. Let $\Ang(B)$ be the 
self-adjoint operator
\begin{equation}
\label{eq:angb}
\Ang(B)=\Bigl(\frac{i}{\Ri}\frac{d}{d\theta}-\frac{B\Ri}{2}\Bigr)^2
\end{equation}
in $L^2\bigl((0,2\pi)\bigr)$ with periodic boundary conditions. Its spectrum 
is easily seen to consist of eigenvalues
$\big\{\bigl(\frac{m}{\Ri}-\frac{B\Ri}{2}\bigr)^2\big\}_{m\in\mathbb{Z}}$. 
In particular
\[
\eigone{\Ang(B)}=\min_{m\in\mathbb{Z}}
\Bigl(\frac{m}{\Ri}-\frac{B\Ri}{2}\Bigr)^2.
\]

Our next theorem states that $\eigone{\Ham(B)}$ will tend to 
$\eigone{\Ang(B)}$ as $\Ro\searrow\Ri$.

\begin{thm}
\label{thm:main2}
Let $B>0$. Then
\[
\lim_{\Ro\searrow\Ri}\eigone{\Ham(B)} 
= \eigone{\Ang(B)}
= \min_{m\in\mathbb{Z}}\Bigl(\frac{m}{\Ri}-\frac{B\Ri}{2}\Bigr)^2.
\]
\end{thm}

\begin{remark}
As a direct consequence of Theorem~\ref{thm:main2} it is possible to find an 
annulus such that the function $B\mapsto \eigone{\Ham(B)}$ is monotonically 
increasing and decreasing alternatively as many times as desired.
\end{remark}

\begin{remark}
Another direct consequence of Theorem~\ref{thm:main2} is that, although the
diamagnetic inequality tells us that $\eigone{\Ham(B)}>\eigone{\Ham(0)}=0$
for all $B>0$ we can actually get $\eigone{\Ham(B)}$ to be arbitrary close
to zero if $B=2m$, $m=1,2,\ldots$ by choosing $\Ro$ close enough to $\Ri$.
\end{remark}

\begin{remark}
Theorem~\ref{thm:main2} can easily be extended to thin cylinders in three 
dimensions, since the third variable then separates. 
\end{remark}

\subsection{Nonmonotonicity in the annulus}

In this section we will prove the spectral 
results Theorem~\ref{thm:mainlin} and Theorem~\ref{thm:main2}. We will work in 
polar coordinates.

\begin{proof}[Proof of Theorem~\ref{thm:mainlin}]
We recall that here $\Ri=1$. Let
\begin{equation}
\label{eq:pot}
p_{m,B}(r)=\Bigl(\frac{m}{r}-\frac{Br}{2}\Bigr)^2
\end{equation}
denote the potential in the quadratic form $\q_m$ in~\eqref{eq:quad}.

We start by showing that if $\Ro>1$ and $m\in\mathbb{Z}\setminus\{1\}$ then 
\begin{equation}
\label{eq:misone}
\eigone{\Ham_m(1)}>\eigone{\Ham_1(1)}.
\end{equation}

The function $f(r)=p_{m,1}(r)-p_{1,1}(r)$ is positive for 
$r>1$. Indeed, $f(r)=1-m+(m^2-1)/r^2$. If $m\not\in\{0,1\}$ then $f$ is 
decreasing, and $f(r)\geq f(1)=m^2-m >0$. If $m=0$ then $f(r)=1-1/r^2$ which is 
clearly positive for all $r>1$. The inequality~\eqref{eq:misone} follows by a 
comparison of quadratic forms.

Next, we show that if $1<\Ro<\sqrt{2m/B}$ then 
\begin{equation}
\label{eq:misdec}
\frac{d}{d B}\eigone{\Ham_m(B)}<0.
\end{equation}

By perturbation theory it holds that
\begin{equation}
\label{eq:pert}
\frac{d}{d B}\eigone{\Ham_m(B)} 
= \int_{1}^{\Ro}\Bigl(\frac{Br^2}{2}-m\Bigr)u(r)^2 r\, dr,
\end{equation}
where $u$ denotes the eigenfunction corresponding to $\eigone{\Ham_m(B)}$.
Moreover the factor $\bigl(\frac{Br^2}{2}-m\bigr)$ is negative for all $1<r<\Ro$ 
if $\Ro<\sqrt{2m/B}$. Inserting this into~\eqref{eq:pert} gives~\eqref{eq:misdec}

It is now easy to finish the proof of Theorem~\ref{thm:mainlin}. Let 
$1<\Ro<\sqrt{2}$. Inequality~\eqref{eq:misone} and analytic perturbation theory 
imply that
\[
\eigone{\Ham(B)}=\eigone{\Ham_1(B)}
\]
for $B$ in a neighborhood of $1$. Since, by~\eqref{eq:misdec}, it holds
that the derivative of $\eigone{\Ham_1(B)}$ is negative at $B=1$ the same is
true for the derivative of $\eigone{\Ham(B)}$. By continuity of the derivative
this holds in a neighborhood of $B=1$. In particular we conclude that the 
function $B\mapsto\eigone{\Ham(B)}$ is strictly decreasing for these values 
of $B$.
\end{proof}

\begin{proof}[Proof of Theorem~\ref{thm:main2}]
Since
\[
\eigone{\Ham(B)}=\inf_m \eigone{\Ham_m(B)},
\]
Theorem~\ref{thm:main2} is a direct consequence of the fact that,
for $m\in\mathbb{Z}$ and $B\geq 0$,
\begin{equation}
\label{eq:limRoRi}
\lim_{\Ro\searrow\Ri} \eigone{\Ham_m(B)} 
= \Bigl(\frac{m}{\Ri}-\frac{B\Ri}{2}\Bigr)^2.
\end{equation}

To get an upper bound we use a trial state. In fact, we use the simplest possible
one. Let $u=\sqrt{2/(\Ro^2-\Ri^2)}$. Then $\|u\|_{L^2((\Ri,\Ro), rdr)}=1$. A 
simple calculation shows that
\[
\begin{aligned}
\lim_{\Ro\searrow\Ri}\q_m[u] 
&= \lim_{\Ro\searrow\Ri}
\biggl(
\frac{2m^2}{\Ro+\Ri}\frac{\log\Ro-\log\Ri}{\Ro-\Ri}
-Bm
+\frac{B^2}{8}\bigl(\Ri^2+\Ro^2\bigr)
\biggr)\\
& = \Bigl(\frac{m}{\Ri}-\frac{B\Ri}{2}\Bigr)^2.
\end{aligned}
\]
Hence $\lim_{\Ro\searrow\Ri} \eigone{\Ham_m(B)} 
\leq \bigl(\frac{m}{\Ri}-\frac{B\Ri}{2}\bigr)^2$.

The lower bound is obtained by using the potential $p_{m,B}(r)$.
Let $u$ be a normalized eigenfunction corresponding to $\eigone{\Ham_m(B)}$. 
then
\[
\eigone{\Ham_m(B)} = \q_m[u] \geq 
\int_{\Ri}^{\Ro} \Bigl(\frac{Br}{2}-\frac{m}{r}\Bigr)^2|u|^2 r\,dr
\geq \min_{\Ri\leq r\leq \Ro}\Bigl(\frac{Br}{2}-\frac{m}{r}\Bigr)^2.
\]
Since $\min_{\Ri\leq r\leq \Ro}\bigl(\frac{Br}{2}-\frac{m}{r}\bigr)^2
\to \bigl(\frac{m}{\Ri}-\frac{B\Ri}{2}\bigr)^2$ as 
$\Ro\searrow\Ri$ we conclude that
\[
\lim_{\Ro\searrow \Ri} \eigone{\Ham_m(B)} 
\geq \Bigl(\frac{m}{\Ri}-\frac{B\Ri}{2}\Bigr)^2.
\]
This completes the proof of~\eqref{eq:limRoRi}, and thus 
finishes the proof of Theorem~\ref{thm:main2}.
\end{proof}

\subsection{Application to the Ginzburg--Landau functional}

In this section we prove Theorem~\ref{thm:main3}. We recall the reader 
that $D$ below denotes the disc with radius $\Ro$, centered at the 
origin. We need the following lemma, and refer to~\cite{fohebook} for 
its proof. 

\begin{lemma}\label{lem:curl}
Let $\Ri$ be fixed and let $\Ri\leq \Ro \leq 2$. There exists a constant 
$\widehat{C}>0$ (independent of $\Ro$) such that for all 
$\mathbf{a}\in H^1_{\Div}(D)$ we have
\[
\|\mathbf{a}\|_{L^2(D)}\leq 
\widehat{C}\|\curl\mathbf{a}\|_{L^2(D)}.
\]
Combining this with the Sobolev embedding we get the existence of a constant 
$\widehat{C}_0$ (independent of 
$\Ro \in [\Ri,2]$) such that for all $\mathbf{a}\in H^1_{\Div}(D)$
\begin{equation}
\label{eq:Sobolev}
\|\mathbf{a}\|_{L^4(D)}\leq 
\widehat{C}_0\|\curl\mathbf{a}\|_{L^2(D)}.
\end{equation}
\end{lemma}

\begin{proof}[Proof of Theorem~\ref{thm:main3}]
Given $0<\epsilon<1$, the Cauchy inequality implies that
\[
|(i\nabla+\kappa \sigma \mathbf{A})\psi|^2
\geq (1-\epsilon)|(i\nabla+\kappa \sigma \mathbf{F})\psi|^2
- \epsilon^{-1} (\kappa \sigma )^2|\mathbf{A}-\mathbf{F}|^2|\psi|^2,
\]
and so
\begin{align}
\mathcal{G}_{\kappa,\sigma }(\psi,\mathbf{A}) &\geq 
\int_{\Omega}(1-\epsilon)|(i\nabla+\kappa \sigma \mathbf{F})\psi|^2 
-\kappa^2|\psi|^2+\frac{\kappa^2}{2}|\psi|^4\, dx\nonumber \\
&\quad- \epsilon^{-1} (\kappa \sigma)^2
\int_{\Omega}|\mathbf{A}-\mathbf{F}|^2|\psi|^2\,dx 
+(\kappa \sigma)^2\int_{D}|\curl\mathbf{A}-1|^2\, dx \nonumber \\
&\geq \bigl((1-\epsilon) \eigone{\Ham(\kappa \sigma)} 
- \kappa^2\bigr) \|\psi \|_{L^2(\Omega)}^2\nonumber \\
&\quad
- \epsilon^{-1} (\kappa \sigma)^2 \| \mathbf{A} 
- \mathbf{F} \|_{L^4(D)}^2 \| \psi \|_{L^4(\Omega)}^2
+(\kappa \sigma)^2\int_{D}|\curl\mathbf{A}-1|^2\, dx \nonumber \\
&\geq
\bigl((1-\epsilon) \eigone{\Ham(\kappa \sigma)} 
- \kappa^2\bigr) \|\psi \|_{L^2(\Omega)}^2\nonumber \\
&\quad
+(\kappa \sigma)^2 
\Bigl( 1 - \widehat{C}_0^2 \epsilon^{-1} \sqrt{\pi} (\Ro^2-\Ri^2)^{1/2}\Bigr) 
\int_{D}|\curl\mathbf{A}-1|^2\, dx.
\end{align}
Here we used~\eqref{eq:Sobolev} and $\|\psi\|_\infty\leq 1$ to get the last 
inequality.

If we choose $\epsilon = (\Ro-\Ri)^{1/4}$, then we see that if 
$\eigone{\Ang(\kappa \sigma)} > \kappa^2$, then for all $\Ro$ sufficiently close 
to $\Ri$ and all $(\psi, \mathbf{A})$,
\begin{align}\label{eq:trivial}
\mathcal{G}_{\kappa,\sigma }(\psi,\mathbf{A}) &\geq 
0.
\end{align}

On the other hand, if $\eigone{\Ham(B = \sigma \kappa)}<\kappa^2$, then
we have (with $\mathbf{F}=1/2(-x_2,x_1)$ and $u$ the normalized eigenfunction
corresponding to $\eigone{\Ham(\sigma\kappa)}$)
\begin{align}\label{eq:Nontrivial}
\mathcal{G}_{\kappa,\sigma}(cu,\mathbf{F}) 
= c^2(\eigone{\Ham(\sigma \kappa)}-\kappa^2)
+c^4\frac{\kappa^2}{2}\int_\Omega |u|^4\,dx < 0
\end{align}
for sufficiently small values of $c$.

Therefore, by the explicit spectrum of $\Ang(B)$ we can choose $\kappa_0>0$ 
and $B_0<B_1<B_2$ such that
\[
\eigone{\Ang(B_j)} < \kappa_0^2, \qquad j=0,2,\qquad
\eigone{\Ang(B_1)} > \kappa_0^2.
\]
Define $\sigma_j := B_j/\kappa_0$.
By the convergence of the spectrum given in Theorem~\ref{thm:main2} 
and~\eqref{eq:Nontrivial} we find the existence of $\widetilde R>\Ri$ such that 
${\mathcal G}_{\kappa_0, \sigma_j}$ has a non-trivial minimizer for all 
$\Ri<\Ro \leq \widetilde R$ and $j\in \{0,2\}$.
On the other hand, it follows from \eqref{eq:trivial} that the minimizer of 
${\mathcal G}_{\kappa_0, \sigma_1}$ is trivial for all $\Ro>\Ri$ sufficiently 
close to $\Ri$.

We conclude the existence of $\Ro>\Ri$ such that there exist non-trivial minimizers 
when $\sigma=\sigma_0$ and $\sigma=\sigma_2$ but not when $\sigma=\sigma_1$. 
Since $\sigma_0<\sigma_1<\sigma_2$ it is clear that $\mathcal{N}(\kappa_0)$ is 
not an interval.
\end{proof}

\section{The case of the complement of the disc}
\label{sec:section4}
\subsection{Introduction}
In this section we consider the case $\Omega=\{x\in\R^2~:~|x|>1\}$ and assume that the magnetic field satisfies Assumption~\ref{ass:Magnetic} with $\delta>0$.
Our aim is to prove Theorem~\ref{thm:extasymptot}.

\subsection{Localization estimate}
Before continuing we give an Agmon estimate for the lowest eigenfunction.

\begin{prop}
\label{prop:Agmon_Outside}
Assume that $\beta$ satisfies Assumption~\ref{ass:Magnetic} with $\delta>0$.
Let $t \in (0,1)$. Then there exist positive 
constants $C$, $a$ and $B_0$ such that if $B>B_0$, and if $\psi$ is an 
eigenfunction of $\Ham(B)$ corresponding to an eigenvalue 
$\lambda \leq t \delta B$.  then 
\begin{equation}
\label{eq:agmon_surface}
\int_{\{|x|>1\}}\exp\bigl(a B^{1/2}\bigl||x|-1\bigr|\bigr)
\bigl(|\psi|^2 + B^{-1} |(-i\nabla + B\mathbf{F}) \psi|^2 \bigr)\,dx
\leq C \int_{\{|x|>1\}}|\psi|^2\,dx.
\end{equation}
\end{prop}

Theorem~8.2.4 in~\cite{fohebook} gives the same estimate with the
restriction that the domain should be bounded.
However, since we give a similar Agmon estimate in Section~\ref{sec:planeg0} 
with proof we omit the proof here.

\subsection{A detailed expansion}
We recall that the quadratic form after decomposition is given by (with $a(r)$ from \eqref{eq:Def_a})
\[
\q_m[u]=\int_{1}^{+\infty}\Bigl(|u'(r)|^2+
\Bigl(\frac{m}{r}-B a(r)\Bigr)^2|u(r)|^2\Bigr)r\,dr.
\]
Notice that at $r=1$ the potential takes the value
\[
\Bigl(\frac{m}{r}-B a(r)\Bigr)^2\Big|_{r=1} 
= (m-\Phi B)^2.
\]
This suggests that we will find the lowest energy for 
$m\approx \Phi B$. That this is the case is the content of the following Lemma.

\begin{lemma}\label{lem:RestrictionOnM}
Let $t \in (0,1)$.
Suppose $\psi = u_m e^{-im \theta}$ is an eigenfunction of $\Ham(B)$ with 
eigenvalue $\lambda \leq t \delta B$. Then 
\[
m = \Phi B + {\mathcal O}(B^{1/2}).
\]
\end{lemma}

\begin{proof}
We neglect the kinetic energy in the expression for $\q_m$. 
Recall the calculation~\eqref{eq:Cancellation}.
For $1<r<2$, we get
\begin{align}
\label{eq:intbetabound}
\Bigl| \int_0^{r-1} (1+s) \tilde\beta(1+s)\,ds \Bigr| \leq C (r-1),
\end{align}
so, estimating the quadratic form with the potential, 
combining~\eqref{eq:intbetabound} and~\eqref{eq:rar-split}, and using 
Proposition~\ref{prop:Agmon_Outside}, we get
\begin{align}\label{eq:quadratic}
 \q_m[u_m] & \geq
 \int_1^2 \frac{1}{r^2}\bigl(m-B r a(r)\bigr)^2 |u_m(r)|^2 r\, dr\\
 &\geq \int_1^2 \frac{1}{r^2}\Bigl[\frac{1}{2}(m-\Phi B)^2-(CB)^2(r-1)^2\Bigr]|u_m(r)|^2 r\, dr\\
 &\geq \frac{1}{8}(m-\Phi B)^2\bigl(1+\mathcal{O}(B^{-\infty})\bigr)
-\widetilde{C}B,
\end{align}
from which the lemma follows.
\end{proof}

\begin{lemma}
\label{lem:uniqueeigenvalue}
Let $t \in (0,1)$. There exists $B_0>0$ such that if $m \in {\mathbb Z}$ and 
$B\geq B_0$, then $\Ham_m(B)$ admits at most one eigenvalue below $t \delta B$.
\end{lemma}

\begin{proof}
Fix $\tilde{t}$ with $t<\tilde{t}<1$. By the lower bound 
\eqref{eq:quadratic}, we see that there exist $B_0, C_0 >0$ 
such that if $|m - \Phi B| \geq C_0 B^{1/2}$, then $\q_m 
\geq \tilde{t}\delta B$. 

So we will 
restrict attention to $m$'s such that $m = \Phi B + \Delta m$,
with $|\Delta m| \leq C_0 B^{1/2}$. Suppose, to get a contradiction, 
that $u_1, u_2$ are eigenfunctions of $\q_m$ corresponding to eigenvalues 
below $t \delta B$.

We write
\begin{align*}
\Bigl(\frac{m}{r}-B a(r)\Bigr)^2 = \frac{1}{r^2} 
\bigl( m - B r a(r) \bigr)^2,
\end{align*}
with 
\begin{align*}
r a(r) = \Phi + \delta (r-1) 
+ {\mathcal O}\bigl((r-1)^2\bigr),
\end{align*}
as $r \rightarrow 1$. So
\begin{align*}
|m - B r a(r)| \geq | m - \Phi B - B \delta (r-1) | 
+ {\mathcal O}\bigl(B (r-1)^2\bigr).
\end{align*}
Using the Agmon estimates, this yields the following bound on normalized 
functions $v$ in $\text{span}\{u_1, u_2\}$.
\begin{equation}
\label{eq:qformest}
\begin{aligned}
\q_m[v] &\geq\int_1^{\infty} \Bigl( |v'(r)|^2 + 
\frac{1}{r^2}\bigl(\Delta m - B \delta (r-1) \bigr)^2 |v(r)|^2\Bigr) r\,dr
+ {\mathcal O}(B^{1/2})\\
&=\widetilde{\q}_m[v] + {\mathcal O}(B^{1/2}),
\end{aligned}
\end{equation}
with
\begin{equation}
\label{eq:tildeq}
\widetilde{\q}_m[v] = \int_1^{\infty} 
 |v'(r)|^2 + \bigl(\Delta m - B \delta (r-1) \bigr)^2 |v(r)|^2\,dr.
\end{equation}
By translation and scaling $\widetilde{\q}_m$ is unitarily equivalent to 
(the quadratic form of) a de Gennes operator (see Appendix~\ref{sec:dG}) and 
therefore has spectrum given 
by
\[
B \delta 
\bigl\{ \eig{j}{\dG}(\Delta m/(\delta B)^{1/2})\bigr\}_{j=1}^{+\infty}
\]
Only the first of these $\eig{j}{\dG}$---counted with multiplicity---is below 
$1$ (for some values of $\Delta m/(\delta B)^{1/2}$), so we reach a 
contradiction if we have a subspace of dimension $2$ on which the quadratic 
form is small.
\end{proof}

\begin{lemma}\label{lem:FurtherRestrictionOnM}
Let $M>0$.
Suppose $\Ham_m(B)$ admits an eigenvalue below $\Theta_0 \delta B + M B^{1/2}$. 
Then there exists a constant $C>0$ such that
\begin{equation}
\label{eq:mcond}
\bigl|m - \bigl(\Phi B + \xi_0 (\delta B)^{1/2}\bigr)\bigr|
\leq CB^{1/4}.
\end{equation}
\end{lemma}

\begin{proof}
By Lemma~\ref{lem:RestrictionOnM}, $|m - \Phi B| = {\mathcal O}(B^{1/2})$.
Assuming that $u$ is the eigenfunction corresponding to the unique 
(by Lemma~\ref{lem:uniqueeigenvalue}) eigenvalue
$\lambda$ below $\Theta_0 \delta B + M B^{1/2}$ we can use the estimate
in~\eqref{eq:qformest} to find that
\[
\q_m[u]\geq \widetilde{\q}_m[u] + \mathcal{O}(B^{1/2}),
\]
with $\widetilde{\q}_m$ as in~\eqref{eq:tildeq}. Implementing the change of
variable $r=1+(\delta B)^{-1/2}\rho$, we get (here we write 
$v(\rho)=(\delta B)^{-1/4}u(1+(\delta B)^{-1/2}\rho)$)
\[
\widetilde{\q}_m[v]=\delta B\int_0^{+\infty} 
|v'(\rho)|^2+
\biggl(\rho-\xi_0+\xi_0-\frac{m-\Phi B}{(\delta B)^{1/2}}\biggr)^2 
|v|^2\,d\rho.
\]
We recognize this as the quadratic form for the de~Gennes operator (see
Appendix~\ref{sec:dG}). By noticing that the first eigenvalue 
$\eigone{\dG}(\xi)$ has a quadratic minimum $\Theta_0$ at $\xi_0$ (and using 
the bound on $(m-\Phi B)/(\delta B)^{1/2}$) we find
that there exists a positive constant $C_0$ such that
\[
\widetilde{\q}_m[v]\geq \biggl[\Theta_0\delta B 
+ C_0\delta B\biggl(\xi_0-\frac{m-\Phi B}{(\delta B)^{1/2}}\biggr)^2\,\biggr] \| v \|^2.
\]
The second term above is bounded by some constant times $B^{1/2}$ according to
the assumption. This in turn gives the existence of a positive constant $C$ 
such that~\eqref{eq:mcond} holds.
\end{proof}

In the remainder of this section we will always restrict our attention to $m$'s 
satisfying the conclusion of Lemma~\ref{lem:FurtherRestrictionOnM}. 

The strategy of the rest of the proof is as follows. We will construct an 
explicit trial state for the operator $\h=\frac1B\Ham_m(B)$ 
(here we suppress the dependence on $m$ and $B$ for the simplicity of notation). 
This trial 
state will be constructed as the first terms of a formal expansion. By taking 
only finitely many terms (for our purposes $3$ terms suffice) and performing a 
localization one gets a well-defined trial state. 
In terms of the objects calculated below our explicit trial state will be as 
follows. Let 
\begin{align}
v(\rho) = v_0 + B^{-1/2} v_1 + B^{-1} v_2,\quad
\lambda = \lambda_0 + B^{-1/2} \lambda_1 + B^{-1} \lambda_2.
\end{align}
Let furthermore, $\chi\in C_0^{\infty}({\mathbb R})$, with $\chi(0) =1$, and 
define (with suitable $\epsilon$, say $\epsilon = (100)^{-1}$)
\begin{align}\label{eq:explicitQuasimode}
\tilde v(r) = (\delta B)^{1/4} \chi(B^{1/2-\epsilon} (r-1)) v( (\delta B)^{1/2} (r-1)).
\end{align}
Then $\| \tilde v \|_{L^2} = 1 + {\mathcal O}(B^{-1/2})$ and
\begin{align}
\| (\h - \lambda) \tilde v \| = {\mathcal O}(B^{-3/2}).
\end{align}
By self-adjointness of $\h$ we get that 
$\dist(\lambda, \sigma(\h)) = {\mathcal O}(B^{-3/2})$. Since we by
Lemma~\ref{lem:uniqueeigenvalue} know that $\h$ has at most
one eigenvalue near $\lambda_0=\delta\Theta_0$, we can conclude 
that $\lambda$ gives 
the first terms of the asymptotic expansion of that lowest eigenvalue of~$\h$.

We proceed to the termwise construction of the trial state.
Since (by Proposition~\ref{prop:Agmon_Outside}) we have localization around 
$r=1$, we implement unitarily the change of variables 
\[
\rho=(\delta B)^{1/2}(r-1),\quad r = 1 +(\delta B)^{-1/2}\rho.
\]
Here, the $\delta$ is included for convenience. Then

\begin{equation}
\label{eq:apot}
B r a(r) = \Phi B + (\delta B)^{1/2} \rho + \frac{1}{2} \rho^2 
+ \frac{k}{6 \delta^{3/2}} B^{-1/2} \rho^3 + {\mathcal O}(B^{-1}).
\end{equation}
Here the estimate on the remainder should be understood in the following sense: We will only act with our operator on the function $\tilde v$ from \eqref{eq:explicitQuasimode} which is localized near $r=1$ on the scale $B^{-1/2}$. So we may consider $\rho$ as a quantity of order $1$.

By Lemma~\ref{lem:FurtherRestrictionOnM} the constant term $m - \Phi B$ 
vanishes to leading order.
For reasons of expositions we will write
\[
m = \Phi B + \mu_1B^{1/2}+\mu_2,
\]
and not insert the choice $\mu_1 = \xi_0 \delta^{1/2}$ until later. Recall 
that $\mu_2 B^{-1/4}$ is bounded.
Integrating by parts, we find (with 
$v(\rho)=(\delta B)^{-1/4}u(1+(\delta B)^{-1/2}\rho)$)
\begin{equation}
\label{eq:diffexp}
\begin{multlined}
\frac{1}{B}\int_1^{+\infty}\Bigl|\frac{d u}{dr} \Bigr|^2\,r\,dr\\
=\delta \int_0^{+\infty} \overline{v} 
\Bigl( - \frac{d^2 v}{d\rho^2} - (\delta B)^{-1/2}  
( 1 + (\delta B)^{-1/2} \rho)^{-1} \frac{dv}{d\rho} \Bigr) 
(1+(\delta B)^{-1/2}\rho)\,d\rho.
\end{multlined}
\end{equation}
We expand our operator $\h$ as
\[
\h = \h_0+B^{-1/2}\h_1+B^{-1}\h_2+\ldots
\]
and obtain
\begin{equation}\label{eq:expanded}
\begin{aligned}
\h_0 &= \delta\Bigl(-\frac{d^2}{d\rho^2}+(\rho-\mu_1/\delta^{1/2})^2\Bigr),\\
\h_1 &= -\delta^{1/2}\frac{d}{d\rho}
-2\mu_2\delta^{1/2}(\rho-\mu_1/\delta^{1/2})
-\frac{2\mu_1^2}{\delta^{1/2}}\rho+3\mu_1\rho^2-\delta^{1/2}\rho^3,\\
\h_2 &= 
\rho\frac{d}{d\rho}
+\mu_2^2
-\frac{4\mu_1 \mu_2}{\delta^{1/2}}\rho 
+3 \mu_2 \rho ^2
+\frac{3 \mu_1^2}{\delta } \rho ^2
-\frac{k \mu_1}{3 \delta ^{3/2}} \rho^3
-\frac{4 \mu_1 }{\delta^{1/2}}\rho ^3
+\frac{k}{3 \delta }\rho ^4
+\frac{5}{4}\rho^4.
\end{aligned}
\end{equation}
We make the Ansatz 
\[
v = \sum_{j=0}^{+\infty} v_j B^{-j/2},\quad 
\lambda = \sum_{j=0}^{+\infty} \lambda_j B^{-j/2}.
\]
Equating order by order in the relation $(\h-\lambda)v=0$ gives:

\paragraph{{\bf Order $B^0$:}} To leading order we find
\[
\h_0 v_0=\lambda_0 v_0,
\]
which is the eigenvalue problem for the de Gennes operator discussed in 
Appendix~\ref{sec:dG}. The optimal eigenvalue $\lambda_0=\delta\Theta_0$ is 
attained for $v_0=\phi_{\xi_0}$ and $\mu_1=\delta^{1/2}\xi_0$.

\paragraph{{\bf Order $B^{-1/2}$:}} Here we get 
\[
(\h_0-\lambda_0)v_1 = (\lambda_1-\h_1)v_0.
\]
By taking scalar product (with measure $d\rho$), we find
\[
0=\langle v_0,(\h_0-\lambda_0)v_1\rangle 
= \lambda_1-\langle v_0,\h_1v_0\rangle.
\]
Via the formulas~\eqref{eq:momentone}--\eqref{eq:dphi} we find 
\begin{equation}
\label{eq:lambdaone}
\begin{aligned}
\lambda_1 &= \langle \phi_{\xi_0},\h_1\phi_{\xi_0}\rangle\\
& = \Bigl\langle \phi_{\xi_0},
\bigl(-\delta^{1/2}\frac{d}{d\rho}
-2\mu_2\delta^{1/2}(\rho-\mu_1/\delta^{1/2})
-\frac{2\mu_1^2}{\delta^{1/2}}\rho+3\mu_1\rho^2
-\delta^{1/2}\rho^3\bigr)\phi_{\xi_0}\Bigr\rangle\\
& = -\delta^{1/2}\langle \phi_{\xi_0},\phi_{\xi_0}'\rangle
-2\mu_2\delta^{1/2}\langle \phi_{\xi_0},(\rho-\xi_0)\phi_{\xi_0}\rangle
-2\xi_0^2\delta^{1/2}\langle \phi_{\xi_0},\rho\phi_{\xi_0}\rangle\\
&\quad+3\xi_0\delta^{1/2}\langle \phi_{\xi_0},\rho^2\phi_{\xi_0}\rangle
-\delta^{1/2}\langle\phi_{\xi_0},\rho^3\phi_{\xi_0}\rangle\\
& = \frac{1}{3}\phi_{\xi_0}(0)^2\delta^{1/2}.
\end{aligned}
\end{equation}

In particular $\lambda_1$ is independent of $\mu_2$. Moreover, since we can 
choose $v_1\perp v_0$, we can let $v_1$ be the regularized resolvent 
$(\h_0-\lambda_0)^{-1}_{\text{reg}}$ of $-\h_1v_0$.
This regularized resolvent is defined as the inverse of the operator 
$(\h_0-\lambda_0)$ restricted to the space $\{v_0\}^{\perp}$.
So we have,
\begin{align}\label{eq:Solv1}
v_1 = -(\h_0-\lambda_0)^{-1}_{\text{reg}}\bigl[\h_1v_0\bigr].
\end{align}

\paragraph{{\bf Order $B^{-1}$:}} We get
\begin{align}\label{eq:degree2}
(\h_0-\lambda_0)v_2 = (\lambda_2-\h_2)v_0+(\lambda_1-\h_1)v_1.
\end{align}
Taking scalar product with $v_0$ again gives
\[
\lambda_2 = \langle v_0,\h_2v_0\rangle 
+ \langle v_0,(\h_1-\lambda_1) v_1\rangle.
\]
We will not calculate this expression in all detail. We are only interested in the 
dependence on $\mu_2$. An inspection gives that it will be a polynomial of
degree two. We will calculate the coefficient in front of $\mu_2^2$ to see that
it is positive so that $\lambda_2$ has a unique minimum with respect to $\mu_2$.

The term $\langle v_0,\h_2v_0\rangle$ is easily calculated since $\h_2$ 
contains one $\mu_2^2$ only.

For the term $\langle v_0,(\h_1-\lambda_1) v_1\rangle$ we find one $\mu_2$ in
$\h_1$ and therefore also one in $v_1$. The coefficient in front of $\mu_2$ in
that term becomes
\begin{multline*}
\Bigl\langle v_0,-2\delta^{1/2}(\rho-\xi_0) 
\bigl(\h_0-\lambda_0\bigr)^{-1}_{\text{reg}}
\bigl[2\delta^{1/2}(\rho-\xi_0)v_0\bigr]\Bigr\rangle\\
=
-4\Bigl\langle (\rho-\xi_0)\phi_{\xi_0},
\bigl(\dG(\xi_0)-\Theta_0\bigr)^{-1}_{\text{reg}}
\bigl[(\rho-\xi_0)\phi_{\xi_0}\bigr]\Bigr\rangle.
\end{multline*}
So, the coefficient in front of $\mu_2^2$ in $\lambda_2$ will be
(see~\eqref{eq:seconddG})
\[
1-4\Bigl\langle (\rho-\xi_0)\phi_{\xi_0},
\bigl(\dG(\xi)-\Theta_0\bigr)^{-1}_{\text{reg}}
\bigl[(\rho-\xi_0)\phi_{\xi_0}\bigr]\Bigr\rangle 
= \xi_0\phi_{\xi_0}(0)^2>0.
\]
This means that we can write
\begin{align}\label{eq:lambda_2}
\lambda_2 = 
\xi_0\phi_{\xi_0}(0)^2
\Bigl(\bigl(\mu_2-C_0^{\text{ext}}\bigr)^2+C_1^{\text{ext}}\Bigr),
\end{align}
where $C_0^{\text{ext}}$ and $C_1^{\text{ext}}$ depend only on $k$, $\delta$,
$\xi_0$ and $\phi_{\xi_0}(0)$ (but not on $\Phi$).

We summarize these findings in a Lemma.

\begin{lemma}\label{lem:expansion}
Suppose 
\[
m = \Phi B+ \xi_0\,(\delta B)^{1/2}+\mu_2,
\]
with $\mu_2 = {\mathcal O}(B^{1/4})$. Then
\begin{align}
\eigone{\Ham_m(B)} &= \Theta_0\delta B 
+ \frac{1}{3}\phi_{\xi_0}(0)^2 (\delta B)^{1/2} 
+ \xi_0\phi_{\xi_0}(0)^2
\bigl((\mu_2 - C_0^{\rm ext})^2+C_1^{\text{ext}}\bigr) \nonumber \\
&\quad+\Oh((1+\mu_2^3)B^{-1/2}).
\end{align}
\end{lemma}

\begin{proof}
We have to control the asymptotic expansion in $\mu_2$ subject to the bound 
$|\mu_2| \leq C B^{1/4}$. Define 
\begin{align}
\lambda^{\rm app} =  \lambda_0  + \lambda_1 B^{-1/2} + \lambda_2 B^{-1},
\end{align}
with $\lambda_0, \lambda_1$ being the constants from above and $\lambda_2$ 
being the quadratic function of $\mu_2$ from \eqref{eq:lambda_2}.
We also define an approximate eigenfunction by 
\begin{align}
v = v_0 + B^{-1/2} v_1 + B^{-1} v_2,
\end{align}
with $v_0 = \varphi_{\xi_0}$, $v_1$ given by \eqref{eq:Solv1}
and $v_2$ being given by solving \eqref{eq:degree2}, i.e.
\begin{align}
v_2 = (\h_0-\lambda_0)^{-1}_{\text{reg}}\bigl[ (\lambda_2-\h_2)v_0
+(\lambda_1-\h_1)v_1 \bigr].
\end{align}
Notice from the explicit form of the operators that $v_1$ depends linearly on 
$\mu_2$ and $v_2$ depends quadratically, so $v$ is normalized to leading order.
Also, by the mapping properties of $(\h_0-\lambda_0)^{-1}_{\text{reg}}$ each 
$v_i$ is a smooth, rapidly decreasing function (see 
Lemma~3.2.9 in~\cite{fohebook}).

We can now estimate as follows
\begin{align}
\| ( \h - \lambda^{\rm app} )v \| &\leq \| ( \h_0 + B^{-1/2} \h_1 + 
B^{-1}\h_2 - \lambda^{\rm app} )v \| \nonumber \\
&\quad+ \| (\h - [\h_0 + B^{-1/2} \h_1 + B^{-1}\h_2]) v \|.
\end{align}
By the decay properties of $v$, the last term is bounded by 
$C (1 + \mu_2^2) B^{-3/2}$.
Our choice of $v$ gives that the first term is equal to
\begin{align}
\| B^{-3/2}[(\h_1-\lambda_1) v_2 +(\h_2-\lambda_2) v_1 ]
+ B^{-2} (\h_2 - \lambda_2) v_2 \|, 
\end{align}
which is easily seen to be bounded by ${\mathcal O}(B^{-3/2}(1+ |\mu_2|^3))$.
\end{proof}

\begin{proof}[Proof of Theorem~\ref{thm:extasymptot}]
Using Lemma~\ref{lem:FurtherRestrictionOnM}, Theorem~\ref{thm:extasymptot} 
follows from Lemma~\ref{lem:expansion} by 
the following argument. Notice that the positive quadratic term in 
$(\mu_2 - C_0^{\rm ext})$ dominates the error term $\mu_2^3 B^{-1/2}$ unless 
$\mu_2$ is bounded in which case the dependence on $\mu_2$ in the error term 
disappears. This finishes the proof of Theorem~\ref{thm:extasymptot}.
\end{proof}

\section{The case of the disc}
\label{sec:disc}

In this section we will indicate a similar calculation of the ground state 
eigenvalue in the case of the unit disc, thereby proving 
Theorem~\ref{thm:intasymptot}, i.e. we work on $\Omega=\{x\in\R^2~:~|x|<1\}$ and for a magnetic field satisfying Assumption~\ref{ass:Magnetic}.

We mainly give the results of the calculations referring to the exterior case 
for details. We will 
have exponential localization estimate like the one of 
Proposition~\ref{prop:Agmon_Outside} (with domain of integration being 
$\{|x|< 1\}$, of course). Therefore, also the rough `localization' of the 
relevant angular momenta---Lemma~\ref{lem:RestrictionOnM}---will hold in this 
case as well. So we can proceed to make a change of variable to the region near 
(on the scale $(\delta B)^{-1/2}$ as before) the boundary.

The leading order terms in the expansion of the operator become very similar 
to the case of the exterior of the disc:

\begin{equation*}
\begin{aligned}
\h_0 &= \delta\Bigl(-\frac{d^2}{d\rho^2}
+(\rho+\mu_1/\delta^{1/2})^2\Bigr),\\
\h_1 &= \delta^{1/2}\frac{d}{d\rho}
+2\mu_2\delta^{1/2}(\rho+\mu_1/\delta^{1/2})
+\frac{2\mu_1^2}{\delta^{1/2}}\rho+3\mu_1\rho^2+\delta^{1/2}\rho^3,\\
\h_2 &= 
\rho\frac{d}{d\rho}
+\mu_2^2
+\frac{4\mu_1 \mu_2}{\delta^{1/2}}\rho
+3 \mu_2 \rho ^2
+\frac{3 \mu_1^2}{\delta } \rho ^2
+\frac{k \mu_1}{3 \delta ^{3/2}}\rho^3
+\frac{4 \mu_1 \rho ^3}{\delta^{1/2}}
+\frac{k}{3 \delta }\rho ^4
+\frac{5}{4}\rho^4.
\end{aligned}
\end{equation*}

The same calculations (using the same Ansatz) as in the previous section
show that (with $\mu_1=-\xi_0/\delta^{1/2}$)
\[
\lambda_0 = \Theta_0,\quad 
\lambda_1 = -\frac{1}{3}\phi_{\xi_0}(0)^2\delta^{1/2},\quad
\lambda_2 = 
\xi_0\phi_{\xi_0}(0)^2\Bigl(\bigl(\mu_2-C_0^{\text{int}}\bigr)^2
+C_1^{\text{int}}\Bigr),
\]
for some constants $C_0^{\text{int}}$ and $C_1^{\text{int}}$, depending only
on the spectral parameters and $\delta$.

Thus, Theorem~\ref{thm:intasymptot} follows from calculations/arguments 
completely analogous to the ones in Section~\ref{sec:section4} and we omit the 
details.

\section{(Non)-monotonicity in the disc and its complement}
\label{sec:nonmondiscanddiscext}
Using the results of Theorem~\ref{thm:extasymptot} and~\ref{thm:intasymptot} 
it is now easy to prove Theorem~\ref{thm:main}.

\begin{proof}[Proof of Theorem~\ref{thm:main}]
We only consider the case of the disc, the complement of the disc being 
similar (using Theorem~\ref{thm:extasymptot} instead of 
Theorem~\ref{thm:intasymptot}).

Assume first that
\begin{align}
\Phi > \frac{\Theta_0}{\xi_0 \phi_{\xi_0}(0)^2} \delta.
\end{align}
Denote by $f$ the function
\[
f(B) = \Phi B-\xi_0(\delta B)^{1/2}+C_0^{\text{int}}.
\]
Notice that $B \mapsto f(B)$ is increasing for all large vaues of $B$.
Choose a sequence $\{B_1^{(n)}\}$ such that 
$f(B_1^{(n)})=n+1/2$, i.e. is a half-integer. 
Let $\varepsilon \in (0,\frac{1}{2\Phi})$.
Choose $B_2^{(n)} = B_1^{(n)}+ \varepsilon$.
Then, for all sufficiently large $n$, $n+1/2 < f(B_2^{(n)}) < n+1$.
So $\Delta_{B_1^{(n)}}^{\text{int}}=1/2$ and
\begin{align}
\lim_{n\to +\infty} \Delta_{B_2^{(n)}}^{\text{int}} 
= \lim_{n\to +\infty}  \bigl(n+1 - f(B_2^{(n)})\bigr) 
= \frac{1}{2} - \Phi \varepsilon.
\end{align}
So we get from the eigenvalue asymptotics 
that
\begin{align}
\eigone{\Ham(B_2^{(n)})} - \eigone{\Ham(B_1^{(n)})} 
&= \Theta_0 \delta \bigl(B_2^{(n)}- B_1^{(n)}\bigr)
-\frac{1}{3}\phi_{\xi_0}(0)^2 \delta^{1/2} 
\bigl[(B_2^{(n)})^{1/2}- (B_1^{(n)})^{1/2}\bigr] \nonumber \\
&\quad+ \xi_0\phi_{\xi_0}(0)^2[(1/2 - \Phi \varepsilon)^2-1/4] + o(1) \nonumber \\
&= \Theta_0 \delta \varepsilon - \xi_0\phi_{\xi_0}(0)^2[\Phi \varepsilon - \Phi^2 \varepsilon^2] + o(1),
\end{align}
which is negative for small $\varepsilon$ (and for all sufficiently 
large $n$) since $\Phi > \frac{ \Theta_0}{\xi_0 \phi_{\xi_0}(0)^2}\delta$ 
by assumption.

Suppose now that
\begin{align}\label{eq:LowFluxGivesMonotone}
\Phi < \frac{\Theta_0}{\xi_0 \phi_{\xi_0}(0)^2} \delta.
\end{align}
We restrict attention to the interval near infinity on which $f(B)$ is increasing. Here we can calculate the right-hand derivative
\begin{align}
\frac{d}{dB}_{+} ( \Delta_{B}^{\text{int}})^2 = \begin{cases} 2 \Delta_B^{\text{int}} f'(B), & \text{if }f(B) \in {\mathbb Z} + [0,1/2),\\
- 2\Delta_B^{\text{int}} f'(B),& \text{if }f(B) \in {\mathbb Z} + [1/2,1).
\end{cases}
\end{align}
So we see that for any $\eta>0$ there exists $B_0>0$ such that for all $\varepsilon>0$ and all $B > B_0$,
\begin{align}\label{eq:RHD_Delta}
( \Delta_{B+\varepsilon}^{\text{int}})^2- ( \Delta_{B}^{\text{int}})^2 
\geq
-2\int_B^{B+\varepsilon}\Delta_{b}^{\text{int}} f'(b)\, db
\geq -(\Phi+\eta) \varepsilon.
\end{align}
We aim to prove monotonicity of $\eigone{\Ham(B)}$, so it suffices to prove 
a positive lower bound on its right hand derivative 
$\frac{d}{dB}_{+}\eigone{\Ham(B)}$, 
which exists by perturbation theory.
Perturbation theory yields, for any $\epsilon>0$,
\begin{align}
\frac{d}{dB}_{+}\eigone{\Ham(B)} 
&= 2 \Re \langle \psi_B, {\bf A} \cdot (-i\nabla + B {\bf A}) \psi \rangle\\
&\geq \frac{\lambda_1(B+\varepsilon) - \lambda_1(B)}{\varepsilon} 
- \varepsilon \int_{\{|x|<1\}} {\bf A}^2 |\psi|^2 \,dx.
\end{align}
Here we completed the square and used the variational characterization of the eigenvalue in order to get the inequality.

Since $\int_{\{|x|<1\}} {\bf A}^2 |\psi|^2 \,dx \leq K$, for some constant 
$K$ independent of $B$, we can estimate, using the eigenvalue asymptotics 
and \eqref{eq:RHD_Delta}
\begin{align} 
\liminf_{B\rightarrow +\infty} \frac{d}{dB}_{+}\eigone{\Ham(B)} 
\geq \Theta_0 \delta - \xi_0\phi_{\xi_0}(0)^2 (\Phi+\eta)
- \varepsilon K.
\end{align}
Since $\varepsilon, \eta$ were arbitrary, we get that
\begin{align} 
\liminf_{B\rightarrow +\infty} \frac{d}{dB}_{+}\eigone{\Ham(B)} 
\geq \Theta_0 \delta - \xi_0\phi_{\xi_0}(0)^2 \Phi.
\end{align}
In particular, $\eigone{\Ham(B)}$ is monotone increasing for large value 
of $B$ if~\eqref{eq:LowFluxGivesMonotone} is satisfied.
\end{proof}

\section{The case of the whole plane with $\delta>0$}
\label{sec:planeg0}

\subsection{Introduction}
In this section we will consider the case $\Omega=\R^2$ and a magnetic field $\beta$ satisfying Assumption~\ref{ass:Magnetic} with $\delta>0$.
We aim to prove Theorem~\ref{thm:wholeplane} for $\delta>0$. This, however,
follows directly once the asymptotic expansion in Theorem~\ref{thm:R2deltapos} 
is obtained, since then it follows that (see~\cite[Section~2.3]{fohebook})
\[
\lim_{B\to+\infty}\frac{d}{dB} \eigone{\Ham(B)} = \delta.
\]
The proof of Theorem~\ref{thm:R2deltapos} follows the same idea as the proof
of Theorem~\ref{thm:extasymptot}. We use a localization of the ground
state to restrict the situation to certain values of the angular momentum. Then
we show that if we find a trial state with low enough energy, it must be related
to the ground state energy. Finally we expand our operator formally and 
construct a trial state that has the correct energy.

\subsection{Agmon estimate for $\delta\geq 0$}
\label{sec:agmonge0}
We start with a localization estimate valid for $\delta \geq 0$. For 
$\delta=0$ it gives the right length scale of the localization. 

\begin{prop}\label{prop:firstagmon}
Suppose $\beta$ satisfies Assumption~\ref{ass:Magnetic} with $\delta \geq 0$.
Let $\psi$ be an eigenfunction of $\Ham(B)$ corresponding to an eigenvalue 
$\lambda \leq \delta B+\omega B^{1/2}$ for some $\omega>0$. Then there exist 
positive constants $C$ and $B_0$ such that
\begin{equation}
\label{eq:agmon}
\int_{\R^2}\exp\bigl(2B^{1/4}\bigl|1-|x|\bigr|\bigr)|\psi|^2\,dx
\leq C \int_{\R^2}|\psi|^2\,dx
\end{equation}
and
\begin{align}\label{eq:AgmonGradient}
\int_{\R^2}\exp\bigl(2B^{1/4}\bigl|1-|x|\bigr|\bigr)
|(-i\nabla + B {\mathbf F})\psi|^2\,dx
\leq C (\delta B + B^{1/2})\int_{\R^2}|\psi|^2\,dx
\end{align}
if $B>B_0$.
\end{prop}

By the localization estimates of Proposition~\ref{prop:firstagmon}, the 
quadratic forms $\q_m$ are well approximated by harmonic oscillators, whose 
ground state eigenvalues are simple. This implies simplicity of the low-lying 
eigenvalues of $\Ham_m(B)$.

\begin{lemma}
\label{lem:uniqueeigenvalue2}
Let $\delta >0$.
Let $\omega > 0$. There exists $B_0>0$ such that if $m \in {\mathbb Z}$ and 
$B\geq B_0$, then $\Ham_m(B)$ 
admits at most one eigenvalue 
below $\delta B + \omega B^{1/2}$.
\end{lemma}

The proof of Lemma~\ref{lem:uniqueeigenvalue2} is similar to that of 
Lemma~\ref{lem:uniqueeigenvalue} and will be omitted.

\begin{proof}[Proof of Prop.~\ref{prop:firstagmon}]
Let $\chi(s)$ be a smooth cut-off function of the real variable $s$ satisfying
\begin{equation}
\label{eq:cutofffunction}
\chi(s)=
\begin{cases}
1, & |s|\leq 1/2,\\
0, & |s|\geq 1,
\end{cases}
\end{equation}
and such that $|\chi'(s)|\leq 3$ for all $s$, and 
$\bigl(1-\chi^2\bigr)^{1/2} \in C^1({\mathbb R})$. Next, let $M$ and $\alpha$ be
positive (to determined below) real numbers and define in $\R^2$ the
functions $\chi_1$ and $\chi_2$ via 
$\chi_1(x)=\chi\bigl(MB^{\alpha}(1-|x|)\bigr)$ and 
$\chi_1(x)^2+\chi_2(x)^2=1$. Then there exists a constant $C_1$ such that
\begin{equation}
\label{eq:nablachi}
\|\nabla \chi_j\|_{\infty} \leq C_1 M B^{\alpha},\quad j\in\{1,2\}.
\end{equation}
Next, for $\ell>0$, let $\Phi_\ell(x)=B^{\sigma}\bigl|1-|x|\bigr|\chi(|x|/\ell)$. 
Then, pointwise in $\R$, it holds that $\Phi_\ell(x)\to B^{\sigma}\bigl|1-|x|\bigr|$ 
as $\ell\to+\infty$. Moreover, $\Phi_\ell$ is differentiable almost everywhere and 
if $\ell\geq 2$ its gradient 
satisfies
\begin{equation}
\label{eq:nablaPhi}
\|\nabla\Phi_\ell\|_{\infty} \leq 4B^{\sigma}.
\end{equation}
Moreover, $\Phi_\ell$ is bounded for all $\ell>0$, so the function 
$\tpsi=\psi e^{\Phi_\ell}$ belongs to the form-domain of $\Ham(B)$.

With the IMS formula, we find that
\begin{equation}
\label{eq:estone}
\begin{aligned}
\q[\chi_1\tpsi]+\q[\chi_2\tpsi]&
\leq \bigl(2C_1M^2B^{2\alpha}+\lambda+16B^{2\sigma}\bigr)\|\tpsi\|^2\\
&\leq \bigl(2C_1M^2B^{2\alpha}
+\delta B+\omega B^{1/2}+16B^{2\sigma}\bigr)\|\tpsi\|^2.
\end{aligned}
\end{equation}
Using that the smallest Dirichlet eigenvalue is greater than the smallest value 
of the magnetic field (again, see~\cite{avhesi}), we find that
\begin{equation}
\q[\chi_1\tpsi]\geq \delta B\|\chi_1\tpsi\|^2
\end{equation}
and
\begin{equation}
\label{eq:qtwoest}
\q[\chi_2\tpsi]\geq 
\Bigl(\delta B+\frac{kB^{1-2\alpha}}{4M^2}\Bigr)\|\chi_2\tpsi\|^2.
\end{equation}
Inserting this into~\eqref{eq:estone} we find that
\begin{equation*}
\delta B\|\tpsi\|^2+\frac{kB^{1-2\alpha}}{4M^2}\|\chi_2\tpsi\|^2
\leq  \bigl(2C_1M^2B^{2\alpha}+\delta B+\omega B^{1/2}
+16B^{2\sigma}\bigr)\|\tpsi\|^2,
\end{equation*}
which can be written
\begin{multline*}
\Bigl(\frac{kB^{1-2\alpha}}{4M^2}-2C_1M^2B^{2\alpha}
-\omega B^{1/2}-16B^{2\sigma}\Bigr)\|\chi_2\tpsi\|^2 \\
\leq 
\bigl(2C_1M^2B^{2\alpha}+\omega B^{1/2}
+16B^{2\sigma}\bigr)\|\chi_1\tpsi\|^2.
\end{multline*}
Choosing $\alpha=\sigma=\frac{1}{4},$ we find that all $B$s factor out, and
hence
\begin{equation*}
\Bigl(\frac{k}{4M^2}-2C_1M^2-\omega-16\Bigr)\|\chi_2\tpsi\|^2
\leq 
\bigl(2C_1M^2+\omega+16\bigr)\|\chi_1\tpsi\|^2.
\end{equation*}
With $M$ so small that the left parenthesis above becomes positive, we find that 
there exists a constant $C_2$ such that
\begin{equation}
\label{eq:chitwochione}
\|\chi_2\tpsi\|^2\leq C_2 \|\chi_1\tpsi\|^2.
\end{equation}
On the support of $\chi_1$ it holds that $MB^{1/4}\bigl|1-|x|\bigr|\leq 1$,
and hence
\[
\exp(\Phi_\ell)=\exp\bigl(B^{1/4}\bigl|1-|x|\bigr|\chi(|x|/\ell)\bigr)
\leq \exp\bigl(\chi(|x|/\ell)/M\bigr)\leq \exp(1/M).
\]
Inserting this in~\eqref{eq:chitwochione} above yields
\[
\|\chi_2\tpsi\|^2\leq 
C_2 \exp(2/M)\|\chi_1\psi\|^2\leq  C_2 \exp(2/M)\|\psi\|^2.
\]
Using monotone convergence we find that
\[
\bigl\|\chi_2\exp\bigl(B^{1/4}\bigl|1-|x|\bigr|\bigr)\psi\bigr\|^2\leq 
C_2 \exp(2/M)\|\chi_1\psi\|^2\leq  C_2 \exp(2/M)\|\psi\|^2.
\]
On the other hand, since $MB^{1/4}\bigl|1-|x|\bigr|\leq 1$ on the support 
of $\chi_1$ it is clear that
\[
\bigl\|\chi_1\exp\bigl(B^{1/4}\bigl|1-|x|\bigr|\bigr)\psi\bigr\|^2
\leq \exp(2/M)\|\psi\|^2.
\]
Combining these two last inequalities we find~\eqref{eq:agmon} with
$C=(1+C_2)\exp(1/M)$.

To prove \eqref{eq:AgmonGradient} we essentially only have to reinsert the 
$L^2$-estimate in the previous calculations. By monotone convergence and the 
IMS-formula, we have
\begin{align}
\int_{\R^2}\exp\bigl(2B^{1/4}\bigl|1-|x|\bigr|\bigr)
&|(-i\nabla + B {\mathbf F})\psi|^2\,dx \nonumber \\
&= \lim_{\ell \rightarrow \infty} 
\int_{\R^2}\exp\bigl(2\Phi_\ell\bigr)|(-i\nabla + B {\mathbf F})\psi|^2\,dx 
\nonumber\\
&=\lim_{\ell \rightarrow \infty}\q[\Psi] 
- \int |\nabla \Phi_\ell|^2 |\Psi|^2 \,dx.
\end{align}
The last term is negative, and we can estimate the first term using again the 
IMS-formula and \eqref{eq:estone} as
\begin{align}
\q[\Psi] \leq \q[\chi_1\Psi] + \q[\chi_2 \Psi] 
\leq (\delta B + C_2 B^{1/2}) \| \Psi \|^2
\end{align}
(with $C_2 = 2 C_1 M^2 + \omega + 16$ and using $\alpha = \sigma =1/4$). 
Now~\eqref{eq:AgmonGradient} follows from~\eqref{eq:agmon}.
\end{proof}

With the help of Proposition~\ref{prop:firstagmon}, we now get a first 
control of the involved angular momenta.

\begin{lemma}\label{lem:LocAngMomFirst}
Let $\delta \geq 0$.
Suppose $\psi = u_m e^{-im\theta}$ is an eigenfunction of $\Ham(B)$ 
with eigenvalue below $\delta B + \omega B^{1/2}$. Then
\begin{align}
m = \Phi B + {\mathcal O}(B^{3/4}).
\end{align}
\end{lemma}
 
 The proof of Lemma~\ref{lem:LocAngMomFirst} is similar to the one of 
Lemma~\ref{lem:RestrictionOnM}---taking into account the weaker localization 
given by Proposition~\ref{prop:firstagmon}---and will be omitted.

\subsection{A detailed expansion for $m- \Phi B = {\mathcal O}(B^{1/2})$}
By Lemma~\ref{lem:uniqueeigenvalue2} there is at most one eigenvalue of 
$\Ham_m(B)$ for sufficiently low energy. So it suffices to construct a trial 
state. The trial function (and all its derivatives) will be localized on the 
length scale $B^{-1/2}$ near $r=1$ (see \eqref{eq:TrialStateFinalHorrible} for 
the explicit choice of trial state). Also the function has support away from 
$r=0$. The calculation is slightly different in different regimes of angular 
momenta $m$. In this subsection, we consider angular momenta satisfying that
\begin{align}\label{eq:GoodRestrictionOnm}
| m - \Phi B | \leq M B^{1/2},
\end{align}
(for some fixed $M>0$). The other case, where $M B^{1/2} \leq | m - \Phi B | 
\leq M' B^{3/4}$ is the object of the next subsection.

We will start by doing a formal expansion of the operator 
$\h=\frac{1}{B}\Ham_m(B)$. We write
\begin{equation}
\label{eq:mexp}
m=\Phi B+\mu_1B^{1/2}+\mu_2.
\end{equation}
With the localization of the trial state in mind, we introduce the new variable 
\[
\rho=(\delta B)^{1/2}(r-1).
\]
This leads to the expansion of our operator as in \eqref{eq:expanded} but as operators on $L^2({\mathbb R})$. Since in the present situation we do not have a boundary, we make the further translation $s:= \rho - \mu_1/\sqrt{\delta}$ to find
\[
\h = \h_0+B^{-1/2}\h_1+B^{-1}\h_2+\ldots
\]
where 
\begin{equation*}
\begin{aligned}
\h_0 &= \delta\Bigl(-\frac{d^2}{ds^2}+s^2\Bigr),\\
\h_1 &= -\delta^{1/2}\frac{d}{ds}+
s\delta^{-1/2}\bigl(\mu_1^2-\delta  \bigl(2\mu_2+s^2\bigr)\bigr),\\
\h_2 &=(s+\mu_1\delta^{-1/2})\frac{d}{ds}
+
\mu_2^2
+\frac{\left(-\mu_1^2+3\delta  s^2
+2 \delta^{1/2} \mu_1 s\right)}{\delta }\mu_2 \\
&+\frac{(\mu_1+\delta^{1/2}s)^2\bigl(4ks(\mu_1+\delta^{1/2}s)
+3\delta^{1/2}(\mu_1^2-6\delta^{1/2}s\mu_1+5\delta s^2)\bigr)}
{12\delta^{5/2}}.
\end{aligned}
\end{equation*}
We do the same Ansatz as above and compare order by order:

\paragraph{\bf Order $B^0$:} To leading order we find
\[
\h_0 v_0=\lambda_0 v_0.
\]
Thus, we choose 
\begin{equation}
\label{eq:gaussian}
v_0=\frac{1}{\pi^{1/4}}\exp(-s^2/2)
\end{equation}
as the normalized ground 
state of the harmonic oscillator, and $\lambda_0=\delta$.

\paragraph{\bf Order $B^{-1/2}$:} Here we get 
\[
(\h_0-\lambda_0)v_1 = (\lambda_1-\h_1)v_0.
\]
By taking scalar product (with measure $ds$), we find
\[
0=\langle v_0,(\h_0-\lambda_0)v_1\rangle 
= \lambda_1-\langle v_0,\h_1v_0\rangle.
\]
Since $v_0$ is an even function it holds that $\langle v_0,\h_1v_0\rangle=0$
and thus $\lambda_1=0$. Moreover, since we can choose $v_1\perp v_0$, we can
let $v_1$ be the regularized resolvent 
$(\h_0-\lambda_0)^{-1}_{\text{reg}}$ of $-\h_1v_0$,
\begin{equation}
\label{eq:gaussianv1}
v_1 = -(\h_0-\lambda_0)^{-1}_{\text{reg}}\bigl[\h_1v_0\bigr].
\end{equation}

\paragraph{\bf Order $B^{-1}$:} We get
\[
(\h_0-\lambda_0)v_2 = (\lambda_2-\h_2)v_0+(\lambda_1-\h_1)v_1.
\]
Taking scalar product with $v_0$ again and using the fact that $\lambda_1=0$, 
gives
\[
\lambda_2 = \langle v_0,\h_2v_0\rangle + \langle v_0,\h_1 v_1\rangle.
\]
Now it holds that (remember: $v_0=\frac{1}{\pi^{1/4}}\exp(-s^2/2)$)
\begin{align*}
\langle sv_0'(s),v_0(s)\rangle&=-\frac{1}{2},&
\langle s^2 v_0(s),v_0(s)\rangle & = \frac{1}{2},\\
\langle v_0'(s),v_0(s)\rangle & = 0,&
\langle s^4 v_0(s),v_0(s)\rangle & = \frac{3}{4},\\
\langle s^j v_0(s),v_0(s)\rangle & = 0,\quad j\text{ odd},&
\langle s^6 v_0(s),v_0(s)\rangle & = \frac{15}{8},\\
\end{align*}
and so 
\begin{equation}
\label{eq:oBm11}
\langle v_0,\h_2v_0\rangle = \frac{1}{4 \delta ^2}\mu_1^4
+\frac{2k-3\delta-4\delta\mu_2}{4\delta^2}\mu_1^2
+\mu_2^2+\frac{3}{2}\mu_2+\frac{7}{16}+\frac{k}{4\delta}.
\end{equation}
The term $\langle v_0,\h_1 v_1\rangle$ is more difficult do calculate. 
But noting that
\begin{equation*}
\begin{aligned}
(\h_0-\lambda_0)\frac{1}{2\delta} sv_0&=sv_0,\\
(\h_0-\lambda_0)\Bigl(-\frac{1}{2\delta}sv_0\Bigr) &= v_0',\quad\text{and}\\
(\h_0-\lambda_0)\frac{s(s^2+3)}{6\delta}v_0 &= s^3 v_0,
\end{aligned}
\end{equation*}
we find that
\[
\begin{aligned}
v_1(s) &= -(\h_0-\lambda_0)^{-1}_{\text{reg}}\bigl(\h_1v_0)\\ 
&= (\h_0-\lambda_0)^{-1}_{\text{reg}}\Bigl(\delta^{1/2}v_0'(s)
-\delta^{-1/2}\mu_1^2 sv_0(s)+2\delta^{1/2}\mu_2sv_0(s)
+\delta^{1/2}s^3v_0(s)\Bigr)\\
& = -\frac{1}{2\delta^{1/2}}sv_0-\frac{\mu_1^2}{2\delta^{3/2}}sv_0
+\frac{\mu_2}{\delta^{1/2}}sv_0+\frac{s(s^2+3)}{6\delta^{1/2}}v_0. 
\end{aligned}
\]
A direct calculation shows that
\[
\begin{aligned}
\h_1v_1 &= -\frac{s^2}{2 \delta ^2}\mu_1^4
+\frac{\left(4 s^4+3 (4 \mu_2-1) s^2+3\right)}{6 \delta}\mu_1^2\\ 
&\quad+\frac{1}{6} \left(-6 \mu_2-s^6+(1-8 \mu_2) s^4
-3 \left(4 \mu_2^2-2 \mu_2+1\right)s^2\right)v_0,
\end{aligned}
\]
so, using the relations above, we find that
\begin{equation}
\label{eq:oBm12}
\langle v_0,\h_1 v_1\rangle = -\frac{1}{4 \delta ^2}\mu_1^4
+\frac{(4 \mu_2+3)}{4 \delta }\mu_1^2
-\frac{1}{16} (8 \mu_2 (2 \mu_2 + 3)+7).
\end{equation}
Combining~\eqref{eq:oBm11} and~\eqref{eq:oBm12} we get  
\[
\lambda_2 = \langle v_0,\h_2v_0\rangle + \langle v_0,\h_1 v_1\rangle
= k\Bigl(\frac{1}{2\delta^2}\mu_1^2 + \frac{1}{4\delta}\Bigr).
\]
We see that $\lambda_2$ is minimal when $\mu_1=0$.

\begin{proof}[Proof of Theorem~\ref{thm:R2deltapos}]
Using Proposition~\ref{prop:BetterLocInm} below it suffices to consider angular momenta satisfying \eqref{eq:GoodRestrictionOnm}.

To finish the proof, based on the calculations above, it is sufficient to provide the trial state that 
gives the right energy. This is done as in the case of the exterior of the disc,
see Section~\ref{sec:section4} for the details. 

We write down the trial state (and $\lambda$) for the sake of completeness. From
the calculations above it follows that (here $\mu_1=0$ and $\mu_2$ is bounded)
\[
\lambda = \lambda_0+\lambda_1 B^{-1/2}+\lambda_2 B^{-1} = \delta
+\frac{k}{4\delta}B^{-1}
\]
Let $v_0$ be the gaussian 
given in~\eqref{eq:gaussian}, $v_1$ the function given in~\eqref{eq:gaussianv1} 
and
\[
v_2(s)=(\h_0-\lambda_0)^{-1}_{\text{reg}}\bigl[(\lambda_2-\h_2)v_0 
+ (\lambda_1-\h_1)v_1\bigr].
\]
Next, let
\[
v(s)=v_0 + v_1 B^{-1/2} + v_2 B^{-1}.
\]
With $\chi\in C_0^{\infty}(\R)$ satisfying $\chi(0)=1$ and $\epsilon=1/100$ we
define our trial state $\tilde{v}(r)$ as
\begin{align}\label{eq:TrialStateFinalHorrible}
\tilde{v}(r)
=B^{1/4}\chi(B^{1/2-\epsilon}(r-1))v\bigl((\delta B)^{1/2}(r-1)\bigr).
\end{align}
\end{proof}

\subsection{Exluding large values of $m - \Phi B$}
In this subsection we will make a preliminary calculation to show that the ground state energy of $\Ham(B)$ restricted to angular momentum $m$ is too large, unless $m - \Phi B = {\mathcal O}(B^{-1/2})$.

\begin{lemma}\label{lem:shooting}
Let $C_0>0$, then there exists $C_1>0$ such that
if $|m - \Phi B| \leq C_0 B^{3/4}$, then 
\begin{align}
\dist(\sigma(\Ham_m(B), \delta B + f(\eta)\sqrt{B}) \leq C_1(|\eta| B^{-1/4} + B^{-1/2}).
\end{align}
Here $\eta := \frac{B \Phi - m}{\delta B^{3/4}}$, and
\begin{align}
f(\eta) = 
\frac{1}{2}k\eta^2.
\end{align}
\end{lemma}

From Lemma~\ref{lem:shooting} we can improve the localization in angular momentum.

\begin{prop}\label{prop:BetterLocInm}
Let $\omega>0$. Then there exists $M, B_0>0$ such that if $B\geq B_0$ and $\Ham_m(B)$ has an eigenvalue below $\delta B + \omega$, then
\begin{align}
| m - \Phi B | \leq M B^{1/2}.
\end{align}
\end{prop}

\begin{proof}[Proof of Proposition~\ref{prop:BetterLocInm}]
This follows by combing Lemma~\ref{lem:uniqueeigenvalue2} and \ref{lem:shooting}.
\end{proof}

\begin{proof}[Proof of Lemma~\ref{lem:shooting}]
The proof is by trial state. We will construct a function (see specific choice in \eqref{eq:TrialStateChoice} below) $\phi \in \dom(\Ham_m(B))$ such that $\| \phi \| \approx 1$ and 
\begin{align}\label{eq:TrialStateCalcAgain}
\| (\Ham_m(B) - [ \delta B + f(\eta)\sqrt{B}]) \phi\| \leq C_1 (|\eta| B^{-1/4} + B^{-1/2}).
\end{align}
By the Spectral Theorem, this implies the Lemma, like in Section~\ref{sec:section4}.
The function that we construct will be localized near $r=1$ on the length scale $B^{-1/2}$ (again this is exactly as in Section~\ref{sec:section4}).

We recall that
\begin{align}
\Ham_m(B) & = - \frac{d^2}{dr^2} - \frac{1}{r}\frac{d}{dr} 
+ \frac{1}{r^2}\bigl(m-B ra(r)\bigr)^2.
\end{align}
Here we will need to expand $\tilde{\beta}$ further than the second
derivative, so we use the full expansion of $r a(r)$ from~\eqref{eq:rar_expansion}.

Introducing $\eta$ as in the lemma and $\rho = (r-1+ B^{-1/4} \eta) \sqrt{B}$, 
we find

\begin{align}
\Ham_m(B) &= - B \frac{d^2}{d\rho^2} - \frac{\sqrt{B}}{1- B^{-1/4} \eta + B^{-1/2} \rho} \frac{d}{d\rho} \nonumber \\
&\quad + \frac{1}{(1- B^{-1/4} \eta + B^{-1/2} \rho)^2}\times \\
&\qquad \biggl[ m - B (1- B^{-1/4} \eta + B^{-1/2} \rho) a(1- B^{-1/4} \eta + B^{-1/2} \rho) \biggr]^2.
\end{align}
Since we will only act with $\Ham_m(B)$ on functions which in the $\rho$ variable are Schwartz functions (see specific choice in \eqref{eq:TrialStateChoice} below), we can treat $\rho$ as a quantity of order $1$ (in terms of powers of $B$), and expand

\begin{align}
\Ham_m(B) = B \left( \mathfrak{h}_{m,0} + B^{-1/4} \mathfrak{h}_{m,1} + B^{-1/2} \mathfrak{h}_{m,2} \right) + {\mathcal O}( |\eta| B^{1/2}) + {\mathcal O}(1),
\end{align}
where
\begin{align}
\mathfrak{h}_{m,0} &= - \frac{d^2}{d\rho^2} + \delta^2 \Bigl(\rho + \frac{\eta^2}{2}\Bigr)^2,\\
\mathfrak{h}_{m,1}&= 
\frac{1}{3}\delta\eta^3(3\delta-k)\Bigl(\rho + \frac{\eta^2}{2}\Bigr),\quad\text{and}\\
\mathfrak{h}_{m,2} &= - \frac{d}{d\rho} 
-\delta^2\Bigl(\rho + \frac{\eta^2}{2}\Bigr)^3
+k\delta\eta^2\Bigl(\rho + \frac{\eta^2}{2}\Bigr)^2\\
&\quad+\frac{1}{12}\Bigl(\delta\eta^4(c-7k+15\delta)\Bigr)\Bigl(\rho + \frac{\eta^2}{2}\Bigr)
+\frac{1}{36}(k-3\delta)^2\eta^6.
\end{align}
We choose 
\begin{align}
v_0 &= \Bigl(\frac{\delta}{\pi}\Bigr)^{1/4} \exp\Bigl( - \frac{\delta}{2}\bigl(\rho + \eta^2/2\bigr)^2\Bigr),
\end{align}
which is the normalized ground state eigenfunction of $\mathfrak{h}_{m,0}$ with eigenvalue $\delta$. 

Next, 
\[
\mathfrak{h}_{m,1} v_0 = \frac{1}{3}\delta\eta^3(3\delta-k)(\rho+\eta^2/2)v_0.
\]
Thus, we want to solve
\[
(\mathfrak{h}_{m,0}-\delta) v_1	= -\mathfrak{h}_{m,1} v_0=-\frac{1}{3}\delta\eta^3(3\delta-k)(\rho+\eta^2/2)v_0,
\]
for $v_1$. A calculation shows that (note that 
$\mathfrak{h}_{m,1} v_0$ is the first exited state of $\mathfrak{h}_{m,1}$ with eigenvalue $3\delta$, in particular orthogonal to $v_0$)
\[
v_1 = -\frac{1}{2\delta}\mathfrak{h}_{m,1} v_0=-\frac{1}{6}\eta^3(3\delta-k)(\rho+\eta^2/2)v_0
\]
gives a solution.

Further calculations yields (the $0$ is there since 
$\langle v_0,\mathfrak{h}_{m,1}v_0\rangle = 0$)
\[
\langle v_0,\mathfrak{h}_{m,2} v_0\rangle + \langle v_0,(\mathfrak{h}_{m,1}-0)v_1\rangle
=\frac{1}{2}k\eta^2.
\]

We further choose
\begin{align}
v_2 &= - (\mathfrak{h}_{m,0} - \delta)_{\text{reg}}^{-1} [ (\mathfrak{h}_{m,2} - f(\eta))v_0 + \mathfrak{h}_{m,1} v_0].
\end{align}
With $\phi$ in \eqref{eq:TrialStateCalcAgain} being chosen as
\begin{align}\label{eq:TrialStateChoice}
\phi =( v_0 + B^{-1/4} v_1 + B^{-1/2} v_2) \times \chi(B^{1/2-\varepsilon}(r-1)),
\end{align}
(similarly to \eqref{eq:explicitQuasimode}), it is immediate to verify \eqref{eq:TrialStateCalcAgain}.
\end{proof}

\section{The case of the whole plane with $\delta=0$}
\label{sec:plane0}

Here we will consider the case $\Omega=\R^2$ and a magnetic field $\beta$
satisfying Assumption~\ref{ass:Magnetic}, with $\delta=0$. We recall that
in this section, $k>0$.

By Proposition~\ref{prop:firstagmon} and Lemma~\ref{lem:LocAngMomFirst} we have 
localization of eigenfunctions corresponding to low-lying eigenvalues on the 
length scale $B^{-1/4}$ and to angular momenta 
$m = \frac{B}{12} + {\mathcal O}(B^{3/4})$.

By~\eqref{eq:Cancellation} and~\eqref{eq:rar-split}, for $|r-1|\leq 1$,
\begin{align}\label{eq:TaylorAgain}
\Bigl( \frac{m}{r} - B a(r)\Bigr)^2 &= \frac{1}{r^2}\biggl(m-\Phi B - \frac{Bk}{6}(r-1)^3+B\mathcal{O}(|r-1|^4)\biggr)^2 \nonumber \\
&\geq \frac{1}{2} \frac{(m-B \Phi)^2}{r^2} - C B^2 r^{-2} (r-1)^6.
\end{align}
Invoking the localization estimates we get the following 
strengthening of Lemma~\ref{lem:LocAngMomFirst}.

\begin{lemma}
Let $\delta = 0$.
Suppose $\psi = u_m e^{-im\theta}$ is an eigenfunction of $\Ham(B)$ 
with eigenvalue below $ \omega B^{1/2}$. Then
\begin{align}
m = \Phi B + {\mathcal O}(B^{1/4}).
\end{align}
\end{lemma}

\begin{proof}
The proof follows from inserting \eqref{eq:TaylorAgain} in the formula for the 
quadratic form $\q_m$ and using the decay estimates in 
Proposition~\ref{prop:firstagmon}.
\end{proof}

We also get a similar result to Lemma~\ref{lem:uniqueeigenvalue}.

\begin{lemma}
\label{lem:uniqueeigenvalue3}
Let $\omega < \inf_{\alpha \in {\mathbb R}}\eig{2}{\M(\alpha)}$, with $\M(\alpha)$ 
the operator from Appendix~\ref{sec:Mont}. There exists $B_0>0$ such that if 
$m \in {\mathbb Z}$ and 
$B\geq B_0$, then $\q_m$ admits at most one eigenvalue below 
$(k/2)^{1/2}\omega B^{1/2}$.
\end{lemma}

\begin{proof}
The proof is analogous to that of Lemma~\ref{lem:uniqueeigenvalue}.
By the localization estimates already obtained, we can see that $\q_m$ is 
given---up to a lower order error---by the quadratic form of the operator 
$\h_0$ from \eqref{eq:h_0_Section8} below which can be recognized as the `Montgomery' operator reviewed in 
Appendix~\ref{sec:Mont}.
\end{proof}

So now we are again in a situation where we know that a sufficiently precise 
trial state must give the asymptotics of the ground state energy.
We write
\begin{equation}
\label{eq:mexputandelta}
m=\Phi B+\mu_3B^{1/4}+\mu_4
\end{equation}
where we will keep $\mu_3$ and $\mu_4$ bounded.
We perform the change of variables
\[
\rho=B^{1/4}(r-1).
\]

Integrating by parts, we find (with $v(\rho)=B^{-1/8}u(1+B^{-1/4}\rho)$ and 
assuming that $u$ is supported away from $0$) that
\begin{equation}
\label{eq:diffexputandelta}
\frac{1}{B^{1/2}}\int_{0}^{+\infty}\Bigl|\frac{d u}{dr} \Bigr|^2\,r\,dr
=\int_{- \sqrt[4]{B}}^{+\infty} \overline{v} 
\Bigl( - \frac{d^2 v}{d\rho^2} - B^{-1/4}  
( 1 + B^{-1/4} \rho)^{-1} \frac{dv}{d\rho} \Bigr) (1+B^{-1/4}\rho)\,d\rho.
\end{equation}
We let $\h=\frac{1}{B^{1/2}}\Ham_m(B)$ and make the Ansatz
\[
\h = \sum_{j=0}^{+\infty}\h_jB^{-j/4},
\quad 
\lambda=\sum_{j=0}^{+\infty} \lambda_j B^{-j/4},
\quad\text{and}\quad 
v = \sum_{j=0}^{+\infty} v_j B^{-j/4},
\]
and get (with notation from \eqref{eq:rar_expansion} and where 
$d= \tilde{\beta}^{(4)}(1)$)

\begin{align}
\label{eq:h_0_Section8}
\h_0 & = -\frac{d^2}{d\rho^2}+\Bigl(\frac{k\rho^3}{6}-\mu_3\Bigr)^2,\\
\h_1 & = -\frac{d}{d\rho}
- \Bigl(\frac{k\rho^3}{6}-\mu_3\Bigr)
\Bigl(\frac{(k-c)\rho^4}{12}-2\mu_3\rho+2\mu_4\Bigr) \\
\h_2 & = \rho\frac{d}{d\rho}
+ \mu_4^2-4\mu_3\mu_4\rho+3\mu_3^2\rho^2
+ \frac{1}{12}(5k-c)\mu_4\rho^4\\
&\qquad+\frac{1}{60}(6c-d-30k)\mu_3\rho^5
+\frac{1}{2880}(5c^2-18ck+8dk+45k^2)\rho^8.
\end{align}

Next we compare the powers of $B$.
\paragraph{\bf Order $B^0$:}
We note that, after a scaling, $\h_0$ becomes 
\[
\Bigl(\frac{k}{2}\Bigr)^{1/2}\Bigl[-\frac{d^2}{d\rho^2}
+\Bigl(\frac{\rho^3}{3}-(2/k)^{1/4}\mu_3\Bigr)^2\Bigr]
=\Bigl(\frac{k}{2}\Bigr)^{1/2} \M\bigl((2/k)^{1/4}\mu_3\bigr),
\]
with the notation from Appendix~\ref{sec:Mont}. By the results of the appendix,
the ground state eigenvalue 
$\eigone{\M(\alpha)}$ has a unique 
non-degenerate minimum $\Xi$ at $\alpha=0$. So we take $\mu_3=0$ and 
find that $\lambda_0= (k/2)^{1/2}\eigone{\M(0)}=(k/2)^{1/2}\Xi$. 
We furthermore take $v_0$ to be the ground state eigenfunction of $\h_0$ 
(with $\mu_3=0$).

\paragraph{\bf Order $B^{-1/4}$:}
Here the equation becomes
\begin{align}
(\h_1-\lambda_1) v_0 + (\h_0 - \lambda_0)v_1 =0.
\end{align}
Taking scalar products with $v_0$, we get
\begin{align}
\lambda_1 = \langle v_0, \h_1 v_0\rangle = 0,
\end{align}
where we used that $\mu_3 = 0$ and that $v_0$ is an even function.
We also determine the function $v_1$ as 
\begin{align}
v_1 = - (\h_0 - \lambda_0)^{-1}_{\text{reg}} (\h_1 v_0) 
\end{align}

\paragraph{\bf Order $B^{-1/2}$:}
At this order, we consider the equation
\begin{align}
(\h_2-\lambda_2) v_0 + \h_1  v_1 + (\h_0 - \lambda_0) v_2=0.
\end{align}
Taking scalar products with $v_0$ determines $\lambda_2$,
\begin{align}
\lambda_2= \langle v_0, \h_2 v_0 \rangle + \langle v_0, \h_1 v_1\rangle.
\end{align}
As a function of $\mu_4$ we see that $\lambda_2$ is a polynomial of degree $2$. 
We determine the coefficient to $\mu_4^2$ as
\[
1 - 4 \Bigl\langle \frac{k\rho^3}{6} v_0, 
(\h_0 - \lambda_0)^{-1}_{\text{reg}} \frac{k\rho^3}{6} v_0 
\Bigr\rangle.
\]
From perturbation theory, we recognize this expression as  
$\tfrac{1}{2}\frac{d^2}{d\alpha^2}\eigone{\M(\alpha)}|_{\alpha=0}$, 
which is positive (by 
Theorem~\ref{thm:Non-degenerate} and Proposition~\ref{prop:secdiffzeropos}).

Thus 
\[
\lambda_2(\mu_4)= \frac{c_0}{2} (\mu_4 - C_1)^2 + C_2,
\]
with $c_0>0$ and for suitable constants $C_1, C_2$. We fix
\begin{align}
v_2 = - (\h_0 - \lambda_0)^{-1}_{\text{reg}} \Bigl[ (\h_2-\lambda_2) v_0 + \h_1  v_1 \Bigr].
\end{align}

\begin{proof}[Proof of Theorem~\ref{thm:R2deltazero}]
To complete the proof of Theorem~\ref{thm:R2deltazero} we only need to give the
trial state that gives the right energy. This is done in the same way as it was
done for the complement of the disc in Lemma~\ref{lem:expansion}.
We omit the details, but mention that the trial state is given by (here
$\epsilon=1/100$ and $\chi\in C_0^\infty$ with $\chi(0)=1$)
\[
\tilde{v}(r) = B^{1/8} \chi(B^{\frac{1}{4}-\epsilon}(r-1)) v(B^{1/4}(r-1)),
\]
with $v= v_0 + B^{-1/4} v_1 + B^{-1/2} v_2$ from the calculations above.
\end{proof}

\begin{proof}[Proof of Theorem~\ref{thm:wholeplane}]
From Theorem~\ref{thm:R2deltazero} it follows exactly like in the proof of 
Theorem~\ref{thm:main} that $B \mapsto \eigone{\Ham(B)}$ is not monotone 
increasing on any half-interval of the form $[B_0, \infty)$. This finishes
the proof of Theorem~\ref{thm:wholeplane}.
\end{proof}

\appendix

\section{The de Gennes operator}
\label{sec:dG}
In this section we have collected some known results on the one-dimensional 
self-adjoint operator
\[
\dG(\xi) = -\frac{d^2}{d\rho^2}+(\rho-\xi)^2
\]
in $L^2((0,+\infty))$ with Neumann condition at $\rho=0$.

We denote by $\eigone{\dG(\xi)}$ the lowest eigenvalue of $\dG(\xi)$ and let
$\phi_{\xi}$ denote the (positive, normalized) ground state. 

It is well-known (see for example~\cite{fohebook}) that this eigenvalue has a 
unique minimum 
\[
\Theta_0 = \min_{\xi\in\R}\eigone{\dG(\xi)},
\]
attained at the unique positive 
\[
\xi_0=(\Theta_0)^{1/2}. 
\]
Moreover, this minimum is non-degenerate; its second derivative at this point 
equals $2\xi_0\phi_{\xi_0}(0)^2$. The following momentum formulas hold:

\begin{equation}
\label{eq:momentone}
\begin{aligned}
\langle \phi_{\xi_0},\phi_{\xi_0}\rangle & = 1,
&
\langle \phi_{\xi_0},(\rho-\xi_0)\phi_{\xi_0}\rangle &= 0,\\
\langle \phi_{\xi_0},(\rho-\xi_0)^2\phi_{\xi_0}\rangle 
&= \frac{1}{2}\xi_0^2,\qquad
&
\langle \phi_{\xi_0},(\rho-\xi_0)^3\phi_{\xi_0}\rangle
&= \frac{1}{6}\phi_{\xi_0}(0)^2.
\end{aligned}
\end{equation}
From these formulas we also find
\begin{equation}
\label{eq:momenttwo}
\begin{gathered}
\langle \phi_{\xi_0},\rho\phi_{\xi_0}\rangle = \xi_0,\quad
\langle \phi_{\xi_0},\rho^2\phi_{\xi_0}\rangle = \frac{3}{2}\xi_0^2,
\quad\text{and}\\
\langle \phi_{\xi_0},\rho^3\phi_{\xi_0}\rangle
= \frac{1}{6}\phi_{\xi_0}(0)^2+\frac{5}{2}\xi_0^3.
\end{gathered}
\end{equation}
Moreover, it holds that
\begin{equation}
\label{eq:dphi}
\langle \phi_{\xi_0},\phi_{\xi_0}'\rangle = -\frac{1}{2}\phi_{\xi_0}(0)^2.
\end{equation}

If we denote by $(\dG(\xi_0)-\Theta_0)^{-1}_{\text{reg}}$ the 
regularized resolvent, then a straight forward calculation shows that
\[
(\dG(\xi_0)-\Theta_0)^{-1}_{\text{reg}}
\bigl[(\rho-\xi_0)\phi_{\xi_0}\bigr] = 
-\frac{1}{2}\phi_{\xi_0}'-\frac{1}{4}\phi_{\xi_0}(0)^2\phi_{\xi_0},
\]
and hence (here we use one of the momentum relations above and integration
by parts)
\begin{equation}
\label{eq:seconddG}
\begin{aligned}
1-4\Big\langle (\rho-\xi_0)\phi_{\xi_0},
&(\dG(\xi_0)-\Theta_0)^{-1}_{\text{reg}}
\bigl[(\rho-\xi_0)\phi_{\xi_0}\bigr]\Big\rangle\\
&=
1-4\Big\langle (\rho-\xi_0)\phi_{\xi_0},
-\frac{1}{2}\phi_{\xi_0}'-\frac{1}{4}\phi_{\xi_0}(0)^2\phi_{\xi_0}
\Big\rangle\\
&= 1-4\Big\langle (\rho-\xi_0)\phi_{\xi_0},
-\frac{1}{2}\phi_{\xi_0}'\Big\rangle\\
&=\xi_0\phi_{\xi_0}(0)^2.
\end{aligned}
\end{equation}
In particular this expression is positive.

\section{A Montgomery operator}
\label{sec:Mont}
For $\alpha\in\R$, we define the Montgomery type operator
\[
\M(\alpha)=-\frac{d^2}{d\rho^2}+\Bigl(\frac{\rho^3}{3}-\alpha\Bigr)^2
\]
as a self-adjoint operator on $L^2(\R)$. Let us denote by 
\[
\eig{1}{\M(\alpha)}<\eig{2}{\M(\alpha)}\leq \ldots
\]
the eigenvalues of $\M(\alpha)$, with corresponding eigenfunctions $u_1$, $u_2$ 
and so on.

\begin{thm}[\cite{fope3}]
\label{thm:Non-degenerate}
The function $\alpha\mapsto\eigone{\M(\alpha)}$ has a unique minimum at $\alpha=0$.
Furthermore, the minimum is non-degenerate, i.e. 
\begin{equation}
\label{eq:c0}
c_0:=\frac{d^2}{d\alpha^2}\eigone{\M(\alpha)}\bigr|_{\alpha=0} > 0.
\end{equation}
\end{thm}

We introduce the following notation,
\begin{equation}
\label{eq:DefXi}
\Xi:=\eigone{\M(0)}.
\end{equation}

\begin{remark}
By the estimates in~\cite{fope3} we know that
\[
0.618\approx\frac{\sqrt{5}-1}{2}<\Xi<
\frac{2^{3/2}}{9}
\Bigl(\frac{4\pi^6-210\pi^4+4410\pi^2-26775}{7}\Bigr)^{1/4}\approx 0.664. 
\]
A numerical value of $\Xi$, calculated by V. Bonnaillie-No\"el, is $0.66$.
\end{remark}

\begin{prop}\label{prop:secdiffzeropos}
It holds that
\[
\frac{d^2}{d\alpha^2}\eigone{\M(\alpha)}\bigr|_{\alpha=0} 
= 2-8\Bigl\langle \frac{\rho^3}{3}u_1,
\bigl(\M(0)-\Xi\bigr)^{-1}_{\text{reg}}\frac{\rho^3}{3}u_1\Bigr\rangle.
\]
\end{prop}

\begin{proof}
Perturbation theory.
\end{proof}

\section{Numerical calculations}

The eigenvalues of $\Ham_m(B)$ can be solved explicitly in terms of confluent
hypergeometric functions, and plotted by the computer. Below we include 
a figure with the lowest eigenvalue of the limit operator $\Ang(B)$ and
the lowest eigenvalue of the annulus of inner radius $\Ri=1$ and outer 
radius $\Ro=3/2$.

\begin{figure}[ht]
\centering
\begin{minipage}{.5\textwidth}
  \centering
  \includegraphics[width=.9\linewidth]{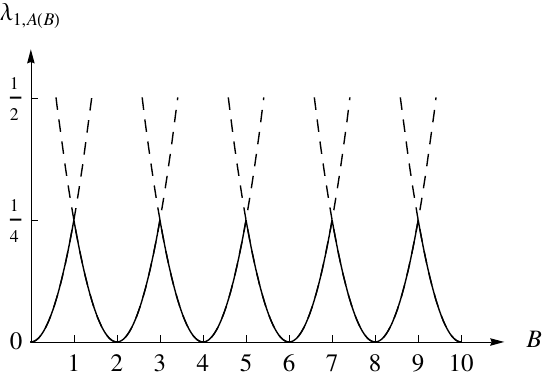}
\end{minipage}%
\begin{minipage}{.5\textwidth}
  \centering
  \includegraphics[width=.9\linewidth]{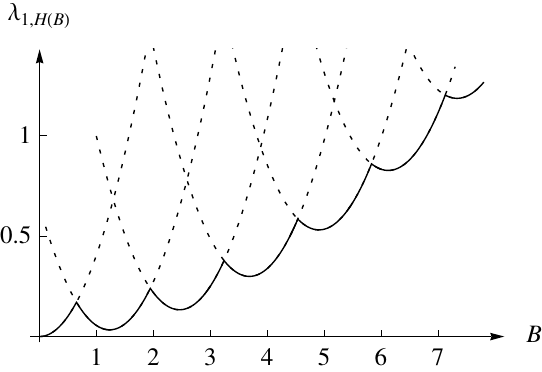}
\end{minipage}
\caption{Left: The eigenvalues of $\Ang(B)$ (dotted) and the lowest eigenvalue 
$\eigone{\Ang(B)}$ (solid) for $0<B<10$ and $\Ri=1$. Right: The lowest eigenvalue $\eigone{\Ham(B)}$ plotted 
for $0<B<8$. The dotted lines are the lowest eigenvalue of $\Ham_m(B)$ for 
$0\leq m\leq 6$, $\Ri=1$ and $\Ro=3/2$.}
\label{fig:angannex}
\end{figure}

\section*{Acknowledgements}
SF was partially supported by the Lundbeck
  Foundation, the Danish Natural Science Research Council and the European 
Research Council under the
 European Community's Seventh Framework Program (FP7/2007--2013)/ERC grant
 agreement  202859.

\def\cprime{$'$}

\end{document}